\documentclass[12pt]{amsart}
\usepackage{amsfonts,amssymb,amsmath,amscd,amsthm, amstext,bm,color}
\usepackage[colorlinks, citecolor=blue,pagebackref,hypertexnames=false]{hyperref}
\usepackage{enumerate}
\usepackage{fancyhdr}
\usepackage{geometry}
\usepackage{graphicx}
\usepackage{epstopdf}
\usepackage{mathrsfs}
\usepackage{txfonts}
\usepackage{tikz}
\usetikzlibrary{arrows.meta, positioning, decorations.pathreplacing}
\usepackage{url}

\usepackage{comment}
\allowdisplaybreaks

\newcommand{\R}{\mathbb R}
\newcommand{\Z}{\mathbb Z}
\newcommand{\Ccal}{\mathcal C}
\newcommand{\Dcal}{\mathcal D}
\newcommand{\Ocal}{\mathcal O}
\newcommand{\Wcal}{\mathcal W}
\newcommand{\Gcal}{\mathcal G}
\newcommand{\one}{\mathbf 1}
\newcommand{\dist}{\operatorname{dist}}
\newcommand{\supp}{\operatorname{supp}}
\newcommand{\co}{\operatorname{co}}
\newcommand{\id}{\operatorname{id}}
\newcommand{\norm}[1]{\left\|#1\right\|}
\newcommand{\abs}[1]{\left|#1\right|}
\newcommand{\ip}[2]{\left\langle #1,#2\right\rangle}

\theoremstyle{plain}
\newtheorem{theorem}{Theorem}[section]
\newtheorem{proposition}[theorem]{Proposition}
\newtheorem{lemma}[theorem]{Lemma}
\newtheorem{corollary}[theorem]{Corollary}

\theoremstyle{definition}
\newtheorem{definition}[theorem]{Definition}
\newtheorem{remark}[theorem]{Remark}
\newtheorem*{introtheoremA}{Theorem A}
\newtheorem*{introtheoremB}{Theorem B}

\numberwithin{equation}{section}

\title{Two-weight inequalities for the Dunkl--Poisson integrals}

\author{Qingdong Guo}
\address{Qingdong Guo, School of Mathematics and Statistics, Xiamen University of Technology, Xiamen 361024, P.R. China}
\email{qingdongmath@126.com}

\author{Ji Li}
\address{School of Mathematical and Physical Sciences, Macquarie University, NSW 2109, Australia}
\email{ji.li@mq.edu.au}

 \author{Brett D. Wick}
\address{Brett D. Wick, Department of Mathematics, Washington University - St. Louis, St. Louis, MO 63130-4899 USA}
\email{wick@math.wustl.edu}

\author{Liangchuan Wu}
\address{Liangchuan Wu,  School of Mathematical Sciences, Anhui University, Hefei, 230601, P.R.~China}
\email{wuliangchuan@ahu.edu.cn}

\date{}

\begin{document}

\begin{abstract}
We prove an $L^2$ two-weight testing theorem for the Dunkl--Poisson semigroup. The difficulty is geometric.  
The Dunkl orbit distance has several reflected diagonals, so a single orbit-box test may mix different chamber components.  
We avoid this by working on one Weyl chamber and keeping the chamber indices.  Under the wall-null assumption the full operator becomes a finite matrix of scalar
positive Poisson-type operators.  
In each entry the orbit diagonal is just the ordinary diagonal in the chamber variables.  
The scalar proof is then a principal-cube stopping-time argument, with two Dunkl kernel comparisons as the only new estimates.  
The resulting forward and backward tests are necessary and sufficient for the original Dunkl--Poisson inequality.
\end{abstract}

\subjclass[2020]{Primary 42B20; Secondary 42B25, 42B35, 43A85.}
\keywords{ Dunkl--Poisson semigroup, two-weight inequality, Weyl chamber, chamber lifting, testing condition, orbit distance.}

\maketitle


\section{Introduction}

A two-weight problem asks when a positive integral operator maps one weighted $L^2$ space into another.  
For the Dunkl--Poisson semigroup the expected answer should look like a testing theorem, but the underlying geometry is not the Euclidean one.  
The kernel is organized by reflection orbits, not by a single singular diagonal.  
Our aim is to put this orbit geometry into variables in which the usual stopping-time proof can be used, without imposing symmetry on the two measures.

Let $R\subseteq \R^N\setminus\{0\}$ be a reduced finite root system, let $G$ be the finite reflection group generated by $R$, and let $\kappa\ge0$ be a
$G$-invariant multiplicity function on $R$.   The associated Dunkl measure is
$$
        d\omega(x):=\prod_{\alpha\in R}\abs{\ip{\alpha}{x}}^{\kappa(\alpha)}\,dx.
$$
The Dunkl--Poisson semigroup has a positive kernel $p_t(x,y)$, symmetric in the variables, $p_t(x,y)=p_t(y,x)$, and $G$-invariant.  
It is recalled in \eqref{eq:poisson-kernel-formula} and acts initially by
$$
        P_t f(x):=\int_{\R^N}p_t(x,y)f(y)\,d\omega(y),\qquad t>0.
$$
Given a locally finite positive Borel measure $\gamma$ on $\R^N$ and a locally finite positive Borel measure $\mu$ on $\R^{N+1}_+:=\R^N\times(0,\infty)$, define
$$
        P_\gamma f(x,t):=\int_{\R^N}p_t(x,y)f(y)\,d\gamma(y),  \qquad (x,t)\in\R^{N+1}_+.
$$
As usual for positive integral operators, the operator is first interpreted on non-negative Borel functions, with values allowed in $[0,\infty]$, 
and on bounded simple functions for norm estimates.  We study the two-weight inequality
\begin{equation}\label{eq:intro-two-weight}
        \norm{P_\gamma f}_{L^2(\R^{N+1}_+,d\mu)}  \le C \norm{f}_{L^2(\R^N,d\gamma)}.
\end{equation}

We work in the Hilbert-space case $p=2$.  The chamber lifting and finite matrix reduction are not specific to $L^2$, 
but a full $L^p$ theorem would require carrying the $p$-dependent scalar stopping-time argument through the chamber notation.  
We keep the present paper at $p=2$ in order to isolate the new Dunkl geometric reduction.

The natural distance in Dunkl analysis is the orbit distance
\begin{equation}\label{eq:intro-orbit-distance}
        d(x,y):=\min_{\sigma\in G}\norm{x-\sigma y}.
\end{equation}
Hence $d(x,y)=0$ exactly on a $G$-orbit.  A full-space proof of \eqref{eq:intro-two-weight} therefore has to deal with the reflected diagonals
$$
        \{(x,y):x=\sigma y\},\qquad \sigma\in G,
$$
all at once.  Ordinary dyadic cubes do not distinguish which reflected diagonal is responsible for a local interaction.  
If all reflected pieces are compressed into a single orbit-box test, then one must compare the masses of the different chamber pieces; 
see Proposition~\ref{prop:orbit-tests-imply-chamber-tests} in Section~\ref{sec:consequences}, especially the comparison hypothesis \eqref{eq:orbit-gamma-comparison}.  
That comparison is useful for orbit-formulated corollaries, but it should not be part of the basic boundedness theorem.

Our main tool is to keep, rather than quotient out, the chamber components. Fix an open fundamental Weyl chamber $\Ccal$, namely a connected component of
$$
        \{x\in\R^N:\ip{x}{\alpha}\ne0\text{ for all }\alpha\in R\}.
$$
We write the elements of the reflection group as
$$
        G=\{\sigma_1,\ldots,\sigma_{|G|}\},
$$
and use this notation throughout to label the chamber pieces
$\sigma_j\Ccal$.

For $x,y\in\Ccal$, we have 
\begin{equation}\label{eq:intro-chamber-identity}
        d(\sigma_\rho x,\sigma_\tau y)=\norm{x-y},  \qquad \sigma_\rho,\sigma_\tau\in G;
\end{equation}
see Lemma~\ref{lem:chamber-minimization}.  Once the chamber variables are fixed, each matrix entry has only the ordinary diagonal $x=y$.

Following the chamber lifting viewpoint used for Dunkl Calder\'on-type commutators in \cite{HanLeeLiSawyerWuCalderonSubmitted}, set
$$
        (U_\gamma f)_\tau(x):=f(\sigma_\tau x),  \qquad  (V_\mu F)_\rho(x,t):=F(\sigma_\rho x,t),  \qquad x\in\Ccal,\,t>0.
$$
Under the wall-null assumption on the two measures, these maps are isometric identifications with finite direct sums over the chambers; see
Lemma~\ref{lem:L2-chamber-isometry}.  The lifted operator is then a finite matrix:
$$
        (V_\mu P_\gamma f)_\rho(x,t)   =  \sum_{\tau=1}^{|G|}T_{\rho\tau}(U_\gamma f)_\tau(x,t),   \qquad (x,t)\in\Ccal\times(0,\infty).
$$
Here
\begin{equation}\label{eq:intro-scalar-entry}
        T_{\rho\tau}h(x,t)   := \int_{\Ccal}   p_t(\sigma_\rho x,\sigma_\tau y)h(y)\,d\gamma_\tau(y).
\end{equation}
Here $\gamma_\tau(E):=\gamma(\sigma_\tau E)$.  
The full two-weight problem is  therefore reduced to a finite-dimensional vector-valued problem whose entries are scalar positive operators on the chamber.  
The orbit non-locality is not ignored; it is encoded in the finite matrix indices $(\rho,\tau)$.

A direct scalar proof over orbit balls would lose this information: one orbit-box can contain mass from several reflected chamber components.  
The chamber proof keeps the components apart.  The wall-null lifting gives finite direct sums over the chambers, 
the identity \eqref{eq:intro-chamber-identity} reduces the singularity in each entry to $x=y$, 
and two kernel comparisons put that entry in the range of the usual positive Poisson-type stopping-time proof.

This way of working is in line with other real-variable arguments in Dunkl analysis, 
where the reflection geometry is kept visible until it becomes a finite-dimensional Euclidean structure.  
Chamber and orbit-adapted decompositions also appear in work on square functions, 
Riesz transforms, and commutators \cite{CGLWW-Square,HanLeeLiSawyerWuCalderonSubmitted, LSWW-Riesz}.
Here it has a useful consequence: the main theorem requires neither $G$-invariance of the measures nor orbit-mass comparison.  
Orbit tests are recovered later only under additional componentwise comparisons; see Proposition~\ref{prop:orbit-tests-imply-chamber-tests} in Section~\ref{sec:consequences}.  
Related two-weight Poisson results include \cite{DLLW,LiWickBessel}, and the testing language follows the classical weighted and two-weight theory
\cite{HolmesLaceyWick,LaceyDuke,LaceySawyerShenUriarte, Muckenhoupt,SawyerMaximal,SawyerPoisson,SawyerShenUriarte,Treil2015}.

We state the main results in the notation used later.  Let $\Dcal$ be the standard dyadic grid in $\R^N$.

For $Q\in\Dcal$ and a fixed dilation factor $a>0$, write
$$
        Q_\Ccal:=Q\cap\Ccal,  \qquad    \widehat Q_\Ccal:=Q_\Ccal\times(0,\ell(Q)],
$$
and
$$
        aQ_\Ccal:=(aQ)\cap\Ccal,   \qquad   \widehat{aQ}_\Ccal:=(aQ_\Ccal)\times(0,\ell(Q)].
$$
Here $aQ$ is the half-open cube with the same centre as $Q$ and side length  $a\ell(Q)$.  
For non-dyadic dilates the boundary convention is immaterial.  For the odd integer dilates used below it fixes the subdivision exactly:
if $a$ is an odd positive integer, then $aQ$ is the disjoint union of the $a^N$ half-open dyadic cubes of side length $\ell(Q)$ contained in $aQ$.
In particular, when $a=3$ we write
$$
        3Q=\bigsqcup_{\theta=1}^{3^N} I_\theta(Q),
$$
where the $I_\theta(Q)$ are these $3^N$ dyadic subcubes.  
We use $\bigsqcup$ only to mark unions that are genuinely disjoint; all other unions are written with $\cup$ or $\bigcup$.

When $Q$ is fixed we suppress it and write simply $I_\theta$.  The vertical height of $\widehat{aQ}_\Ccal$ is always $\ell(Q)$, not $\ell(aQ)$.

Let
\begin{equation}\label{eq:intro-hyperplanes}
     \Wcal:=\bigcup_{\alpha\in R}  \left\{x\in\R^N:\ip{x}{\alpha}=0\right\}
\end{equation}
be the union of the reflection hyperplanes.  For wall-null measures we write
$$
        \gamma_\tau(E):=\gamma\left(\sigma_\tau E\right),  \qquad   \mu_\rho(S):=\mu\left((\sigma_\rho\times\id)S\right),
$$
where $E\subseteq\Ccal$ and $S\subseteq\Ccal\times(0,\infty)$ are Borel sets.  
The  backward tests are understood with non-negative extended integrals.  
If an input  chamber component is the zero measure, the corresponding scalar entry and its testing constants are taken to be zero.  
No separate Carleson-box finiteness  assumption is imposed: once the forward test is finite and the input chamber measure is non-zero, 
Lemma~\ref{lem:forward-implies-square-finiteness} shows that the quantities $\iint_{\widehat Q_\Ccal}t^2\,d\mu_\rho$ are finite on every chamber box needed in the proof.

The wall-null assumption is used only to avoid multiple representatives in the  chamber parametrization. 
It is automatic for the Dunkl measure and for all measures absolutely continuous with respect to Lebesgue measure. 
For measures charging reflection walls, the statement would have to be supplemented by a finite stratification of the walls and by the corresponding lower-dimensional
versions of the chamber parametrization. 
We do not pursue this extension here.

\begin{introtheoremA}
Let $\gamma$ be a locally finite positive Borel measure on $\R^N$, and let $\mu$ be a locally finite positive Borel measure on $\R^{N+1}_+$.  Assume
$$
        \gamma(\Wcal)=0  \quad \text{and} \quad    \mu(\Wcal\times(0,\infty))=0.
$$
Let $T_{\rho\tau}$ be the scalar entries defined by \eqref{eq:intro-scalar-entry}.  Then
$$
        P_\gamma:L^2(\R^N,d\gamma)     \longrightarrow   L^2(\R^{N+1}_+,d\mu)
$$
is bounded \textbf{if and only if}, for every $1\le \rho,\tau\le |G|$,  
the following two testing conditions hold uniformly over all dyadic cubes $Q$  with $Q_\Ccal\ne\varnothing$:
\begin{equation}\tag{A1}\label{eq:intro-forward-test}
        \iint_{\widehat{3Q}_\Ccal}   |T_{\rho\tau}(\one_{Q_\Ccal})(x,t)|^2   \,d\mu_\rho(x,t)   \le   F_{\rho\tau}^2\gamma_\tau(Q_\Ccal)
\end{equation}
and
\begin{equation}\tag{A2}\label{eq:intro-backward-test}
        \int_{3Q_\Ccal}   |T_{\rho\tau}^*(t\one_{\widehat Q_\Ccal})(y)|^2  \,d\gamma_\tau(y)
        \le   B_{\rho\tau}^2   \iint_{\widehat Q_\Ccal} t^2\,d\mu_\rho(x,t),
\end{equation}
where $F_{\rho\tau}$ and $B_{\rho\tau}$ denote the least admissible constants in  these two tests, and where
$$
       T_{\rho\tau}^*g(y)  :=   \iint_{\Ccal\times(0,\infty)}  p_t(\sigma_\rho x,\sigma_\tau y)g(x,t)\,d\mu_\rho(x,t).
$$
Moreover, if $\mathcal N$ denotes the best constant in \eqref{eq:intro-two-weight},
then
$$
        \max_{\rho,\tau}(F_{\rho\tau}+B_{\rho\tau})   \lesssim_{N,R,\kappa}   \mathcal N
        \lesssim_{|G|,N,R,\kappa}   \max_{\rho,\tau}(F_{\rho\tau}+B_{\rho\tau}).
$$
The lower comparison has constants independent of the group order.  
The upper  comparison is the finite matrix estimate and deteriorates with $|G|$.  This is proved below as Theorem~\ref{thm:main-chamber}.
\end{introtheoremA}

The proof of Theorem~A reduces to the following scalar chamber statement, applied to each matrix entry.

\begin{introtheoremB}
Fix one pair $(\rho,\tau)$ with $1\le\rho,\tau\le |G|$.  Write
$$
        Tf(x,t):=\int_\Ccal p_t(\sigma_\rho x,\sigma_\tau y)f(y)\,d\gamma_\tau(y).
$$
Then
$$
        T:L^2(\Ccal,d\gamma_\tau)   \longrightarrow     L^2(\Ccal\times(0,\infty),d\mu_\rho)
$$
is bounded \textbf{if and only if} the two tests \eqref{eq:intro-forward-test}  and \eqref{eq:intro-backward-test}, 
with this fixed pair $(\rho,\tau)$, hold  uniformly over dyadic cubes.  Here $F_{\rho\tau}$ and $B_{\rho\tau}$ denote the least admissible testing constants for this fixed pair.  
The operator norm is  comparable to $F_{\rho\tau}+B_{\rho\tau}$, with constants depending only on $N$, $R$, and $\kappa$, 
and hence independent of the pair of chambers.  The proof is given in Theorem~\ref{thm:scalar-chamber}.
\end{introtheoremB}

The chamber formulation has three useful consequences, developed in detail in Section~\ref{sec:consequences}. 
\begin{itemize}
\item  First, if both weights are $G$-invariant, then the  $|G|^2$ chamber tests of Theorem~A reduce to $|G|$ tests for the relative chamber
operators indexed by $s=\sigma_\rho^{-1}\sigma_\tau\in G$.

\item  Second, if both  measures are supported in single, possibly distinct, chambers, the full inequality reduces to one scalar chamber test.

\item Third, the older orbit-box  testing formulation can be recovered from Theorem~A 
under suitable  componentwise orbit-mass comparison hypotheses; this clarifies the precise place where orbit-mass comparability enters, 
and why no such hypothesis is needed in the chamber theorem itself.
\end{itemize}

The local $t^2\,d\mu$ finiteness needed for the backward tests is not an  additional hypothesis.  
It follows from finite forward testing by Lemma~\ref{lem:forward-implies-square-finiteness}.

The paper is organized as follows.  
Section~\ref{sec:dunkl-preliminaries} recalls the Dunkl objects and the Poisson kernel estimates.  
Section~\ref{sec:chamber-lifting} develops the chamber lifting and proves the orbit-metric identity.  
Section~\ref{sec:scalar-chamber} contains the dyadic decomposition, the elementary Whitney overlap estimates, and the proof of the scalar chamber theorem.  
Section~\ref{sec:matrix-theorem} assembles the finite matrix theorem from the scalar chamber estimates.  
Section~\ref{sec:consequences} derives the invariant, single-chamber, and orbit-box consequences.  
The appendix gives the finite dyadic packing argument used in the stopping-time proof.

In this paper, dyadic cubes are half-open, and dilated cubes are taken with the corresponding half-open convention.  
This gives exact disjoint decompositions even for arbitrary Borel measures.  The convention has no effect on any estimate.  
The notation $f\lesssim g$ and $f\simeq g$ mean that there exists $C\ge1$ such that $f\le Cg$ and $f/C\le g\le Cf$, respectively; 
subscripts on $\lesssim$ indicate the allowed structural dependence of the implicit constant.

\section{Dunkl preliminaries}\label{sec:dunkl-preliminaries}

We recall the notation used throughout the paper.  For background on Dunkl operators, Dunkl transforms, and related applications; see
\cite{AmriHammi, AmriSifi, DaiWang, deJeu, Dunkl1989,Dunkl1991, Opdam}.  
The Euclidean inner product and norm on $\R^N$ are denoted by $\ip{x}{y}$ and $\norm{x}:=(\sum_{i=1}^{N}|x_{i}|^{2})^{1/2}$.

Let $R\subseteq \R^N\setminus\{0\}$ be a reduced root system.  Thus
$$
        R\cap \R\alpha=\{\alpha,-\alpha\}, \qquad \sigma_\alpha(R)=R,
$$
where
$$
        \sigma_\alpha(x):=x-2\frac{\ip{x}{\alpha}}{\norm{\alpha}^2}\alpha.
$$
We normalize the root normals by $\norm{\alpha}^2=2$.  
This choice does not change the reflection or the Dunkl difference quotient, and changes $d\omega$ only by an overall constant.

Let $G$ be the finite reflection group generated by  $\{\sigma_\alpha:\alpha\in R\}$.  
If $E\subseteq \R^N$ and $x\in\R^N$, set
$$
        O(E):=\bigcup_{\sigma\in G}\sigma(E),  \qquad   O(x):=O(\{x\}).
$$
The orbit distance \eqref{eq:intro-orbit-distance} satisfies $\abs{\norm{x}-\norm{y}}\le d(x,y)\le \norm{x-y}$, $x,y\in\R^N$.

Let $\kappa:R\to[0,\infty)$ be $G$-invariant.  The Dunkl operator in direction $\xi\in\R^N$ is
$$
     T_\xi f(x):=\partial_\xi f(x)   +\sum_{\alpha\in R}
        \frac{\kappa(\alpha)}{2}\ip{\alpha}{\xi}   \frac{f(x)-f(\sigma_\alpha x)}{\ip{\alpha}{x}}.
$$
The Dunkl measure is $d\omega(x)=\omega(x)dx$, where
$$
     \omega(x):=\prod_{\alpha\in R}\abs{\ip{\alpha}{x}}^{\kappa(\alpha)}.
$$
It is $G$-invariant.  Put
$$
        \mathbf N:=N+\sum_{\alpha\in R}\kappa(\alpha).
$$
The measure $d\omega$ is doubling in the form
$$
     \lambda^N\omega(B(x,r)) \lesssim  \omega(B(x,\lambda r))  \lesssim
        \lambda^{\mathbf N}\omega(B(x,r)),   \qquad \lambda\ge1.
$$
Since $G$ is finite and $\omega$ is $G$-invariant,
$$
     \omega(B(x,r))   \le \omega(O(B(x,r)))    \le   |G|\,\omega(B(x,r)).
$$

Let $e_1,\ldots,e_N$ be the standard basis and write $T_j:=T_{e_j}$.  Set
$$
        \Delta_\kappa:=\sum_{j=1}^N T_j^2.
$$
Following our convention, the non-negative self-adjoint Dunkl Laplacian is $-\Delta_\kappa$, and the Dunkl--Poisson semigroup is
$$
        P_t=e^{-t\sqrt{-\Delta_\kappa}},    \qquad t>0.
$$
We use the standard positivity, symmetry, $G$-invariance, integral  representation, and pointwise estimates for the Dunkl--Poisson kernel.  
These  facts follow from the positivity of the Dunkl intertwining operator and the  kernel estimates for Dunkl convolution; 
see R\"osler \cite{RoslerPositivity},  R\"osler--Voit \cite{RoslerVoit}, and Thangavelu--Xu \cite[Section~4]{ThangaveluXu}.

The kernel of $P_t$ has the integral representation
\begin{equation}\label{eq:poisson-kernel-formula}
        p_t(x,y)  :=  \int_{\R^N}   \frac{c_{N,\kappa}t}     {\bigl(t^2+A(x,y,\xi)^2\bigr)^{(\mathbf N+1)/2}}    \,d\lambda_x(\xi),
\end{equation}
where $\lambda_x$ is a probability measure supported in $\co(O(x))$ and
$$
     A(x,y,\xi):=  \bigl(\norm{x}^2+\norm{y}^2-2\ip{y}{\xi}\bigr)^{1/2}.
$$
The normalizing constant is
$$
        c_{N,\kappa}    :=2^{\mathbf N/2}\pi^{-1/2}\Gamma\bigl((\mathbf N+1)/2\bigr) \biggl(\int_{\R^N}e^{-\norm{x}^{2}/2}\,d\omega(x)\biggr)^{-1}.
$$
We use the symmetry
\begin{equation}\label{eq:p-symmetry}
        p_t(x,y)=p_t(y,x)>0,\quad x,y\in\R^N,
\end{equation}
the $G$-invariance
\begin{equation}\label{eq:p-G-invariance}
        p_t(\sigma x,\sigma y)=p_t(x,y),   \qquad \sigma\in G,
\end{equation}
and the following mixed-metric upper estimate
\begin{equation}\label{eq:poisson-upper}
        p_t(x,y)    \lesssim     \frac{t}    {\omega(B(x,t+d(x,y)))+     \omega(B(y,t+d(x,y)))}     \frac1{t+\norm{x-y}}.
\end{equation}
This form is intentional.  The ball terms use the orbit distance $d(x,y)$, while  the last factor uses the Euclidean distance $\norm{x-y}$.  
Since  $\norm{x-y}\ge d(x,y)$, the last factor is finer than the orbit-distance version,   and both forms are used below.
The following elementary inequality will be used several times:
\begin{equation}\label{eq:A-lower-d}
        d(x_0,y_0)  \le   A(y_0,x_0,\xi)     =    \bigl(\norm{y_0}^2-\norm{\xi}^2+\norm{x_0-\xi}^2\bigr)^{1/2},    \qquad \xi\in\co(O(y_0)).
\end{equation}
Indeed, since $\xi$ is a convex combination of points in $O(y_0)$,
$$
        \langle x_0,\xi\rangle     \le   \max_{\sigma\in G}\langle x_0,\sigma y_0\rangle.
$$
Therefore
$$
        A(y_0,x_0,\xi)^2    =\norm{x_0}^2+\norm{y_0}^2-2\langle x_0,\xi\rangle    \ge
        \min_{\sigma\in G}\norm{x_0-\sigma y_0}^2    =d(x_0,y_0)^2.
$$

\section{Chamber lifting}\label{sec:chamber-lifting}

Let $\Wcal$ be as in \eqref{eq:intro-hyperplanes}.  The connected components of
$\R^N\setminus\Wcal$ are the open Weyl chambers.  Fix one of them and denote it
by $\Ccal$.  We write the elements of $G$ as
$
        G=\{\sigma_1,\ldots,\sigma_{|G|}\},
$
and use this notation to label the chamber pieces.  Then
$$
     \R^N\setminus\Wcal  =   \bigsqcup_{\tau=1}^{|G|} \sigma_\tau\Ccal.
$$

\begin{definition}\label{def:wall-null}
Let $\gamma$ be a locally finite positive Borel measure on $\R^N$ and let $\mu$ be a locally finite positive Borel measure on $\R^{N+1}_+$.  
We say that $(\gamma,\mu)$ is wall-null if
$$
        \gamma(\Wcal)=0 \quad \text{and} \quad     \mu(\Wcal\times(0,\infty))=0.
$$
\end{definition}

Throughout this section, the pair $(\gamma,\mu)$ is assumed to be wall-null in the sense of Definition~\ref{def:wall-null}.  
For each chamber index $\tau$, we denote by $\gamma_\tau$ the pullback of $\gamma$ to the fixed chamber $\Ccal$, that is,
$$
        \gamma_\tau(E):=\gamma(\sigma_\tau E)    \quad\text{for every Borel set } E\subseteq\Ccal.
$$
Similarly, for each chamber index $\rho$, we denote by $\mu_\rho$ the pullback of $\mu$ to $\Ccal\times(0,\infty)$,
$$
        \mu_\rho(S):=\mu((\sigma_\rho\times\id)S)  \quad\text{for every Borel set } S\subseteq\Ccal\times(0,\infty).
$$

\begin{lemma}\label{lem:L2-chamber-isometry}
Assume that $(\gamma,\mu)$ is wall-null and $\gamma_\tau,\mu_\rho$ are defined as above.  Define
$$
        (U_\gamma f)_\tau(x):=f(\sigma_\tau x),  \qquad x\in\Ccal,
$$
and
$$
        (V_\mu F)_\rho(x,t):=F(\sigma_\rho x,t),   \qquad (x,t)\in\Ccal\times(0,\infty).
$$
Then $U_\gamma$ and $V_\mu$ induce the following isometric identifications:
$$
        L^2(\R^N,d\gamma)   \simeq   \bigoplus_{\tau=1}^{|G|} L^2(\Ccal,d\gamma_\tau)
$$
and
$$
        L^2(\R^{N+1}_+,d\mu)   \simeq   \bigoplus_{\rho=1}^{|G|}    L^2(\Ccal\times(0,\infty),d\mu_\rho).
$$
\end{lemma}

\begin{proof}
The open chambers $\sigma_\tau\Ccal$ are pairwise disjoint and cover $\R^N\setminus\Wcal$.  Since $\gamma(\Wcal)=0$,
\begin{align*}
        \norm{f}_{L^2(\R^N,d\gamma)}^2  &=\sum_{\tau=1}^{|G|}\int_{\sigma_\tau\Ccal}|f(x)|^2\,d\gamma(x)        
        =\sum_{\tau=1}^{|G|}\int_\Ccal |f(\sigma_\tau x)|^2\,d\gamma_\tau(x).
\end{align*}
This proves that $U_\gamma$ is an isometry.  
It is onto the direct sum: given $(h_\tau)_\tau$, define $f(\sigma_\tau x)=h_\tau(x)$ on the open chambers and arbitrarily on $\Wcal$.  
The wall-null assumption makes the definition on $\Wcal$ irrelevant.  
Thus the first identification follows.  
The second one is identical, using the partition of the spatial variable in $\R^{N+1}_+$ and the assumption $\mu(\Wcal\times(0,\infty))=0$:
$$
        \norm{F}_{L^2(\R^{N+1}_+,d\mu)}^2  =  \sum_{\rho=1}^{|G|}
        \iint_{\Ccal\times(0,\infty)}|F(\sigma_\rho x,t)|^2\,d\mu_\rho(x,t).
$$
Surjectivity is obtained in the same way by defining $F(\sigma_\rho x,t)$ from an arbitrary chamber vector and ignoring $\Wcal\times(0,\infty)$.
\end{proof}

The following elementary fact is where the chamber geometry enters.

\begin{lemma}\label{lem:chamber-minimization}
Let $\overline\Ccal$ be the closure of the fixed chamber.  If $x,y\in\overline\Ccal$, then
\begin{equation}\label{eq:chamber-minimization}
        \min_{\sigma\in G}\norm{x-\sigma y}=\norm{x-y}.
\end{equation}
Consequently, for all $x,y\in\Ccal$ and all $1\le\rho,\tau\le |G|$,
\begin{equation}\label{eq:chamber-orbit-identity}
        d(\sigma_\rho x,\sigma_\tau y)=\norm{x-y}.
\end{equation}
\end{lemma}

\begin{proof}
Let the positive system be the one whose closed chamber is $\overline\Ccal$. For $y\in\overline\Ccal$ and $\sigma\in G$, 
the standard dominance property of finite reflection groups gives that $y-\sigma y$ is a non-negative linear combination of the corresponding simple roots.  
Since $x\in\overline\Ccal$, we have $\langle x,y-\sigma y\rangle\ge0$.  Hence
$$
        \norm{x-\sigma y}^2-\norm{x-y}^2  =2\langle x,y-\sigma y\rangle   \ge0,
$$
which proves \eqref{eq:chamber-minimization}.  For the second assertion,
\begin{align*}
        d(\sigma_\rho x,\sigma_\tau y)
        =\min_{\theta\in G}\norm{\sigma_\rho x-\theta\sigma_\tau y}  
        =\min_{\theta\in G}\norm{x-\sigma_\rho^{-1}\theta\sigma_\tau y} 
        =\min_{\sigma\in G}\norm{x-\sigma y}
        =\norm{x-y}.
\end{align*}
The proof is complete.
\end{proof}

For $1\le \rho,\tau\le |G|$, define the lifted scalar kernel
$$
        K_t^{\rho\tau}(x,y) := p_t(\sigma_\rho x,\sigma_\tau y),   \qquad x,y\in\Ccal,   \quad t>0.
$$
The corresponding scalar operator is
$$
       T_{\rho\tau}h(x,t)  = \int_{\Ccal}K_t^{\rho\tau}(x,y)h(y)\,d\gamma_\tau(y).
$$
Its adjoint relative to $d\gamma_\tau$ and $d\mu_\rho$ is
$$
      T_{\rho\tau}^*G(y) = \iint_{\Ccal\times(0,\infty)}  K_t^{\rho\tau}(x,y)G(x,t)\,d\mu_\rho(x,t).
$$
By Tonelli's theorem and the non-negativity of the kernel,
$$
        \iint_{\Ccal\times(0,\infty)} T_{\rho\tau}h(x,t)G(x,t)\,d\mu_\rho(x,t) =
        \int_\Ccal h(y)T_{\rho\tau}^*G(y)\,d\gamma_\tau(y)
$$
for non-negative Borel functions $h$ and $G$.  
Thus, whenever $T_{\rho\tau}$ is bounded between the two $L^2$ spaces, the displayed formula is its Hilbert space adjoint.

\begin{lemma}\label{lem:matrix-kernel-upper}
For every $1\le\rho,\tau\le |G|$ and every $x,y\in\Ccal$, $t>0$,
\begin{equation}\label{eq:matrix-kernel-upper}
        K_t^{\rho\tau}(x,y) \lesssim  \frac{t} {\omega(B(x,t+\norm{x-y}))+
         \omega(B(y,t+\norm{x-y}))} \,\frac1{t+\norm{x-y}}.
\end{equation}
The implicit constant is independent of $\rho$ and $\tau$.
\end{lemma}

\begin{proof}
Apply \eqref{eq:poisson-upper} with the full-space variables $u=\sigma_\rho x$ and $v=\sigma_\tau y$.  By Lemma~\ref{lem:chamber-minimization},
$$
        d(\sigma_\rho x,\sigma_\tau y)=\norm{x-y}.
$$
Also $\omega$ is $G$-invariant, so
$$
        \omega(B(\sigma_\rho x,r))=\omega(B(x,r)),  \qquad \omega(B(\sigma_\tau y,r))=\omega(B(y,r)).
$$
Finally,
$$
        \norm{\sigma_\rho x-\sigma_\tau y}  \ge d(\sigma_\rho x,\sigma_\tau y)    =  \norm{x-y}.
$$
Substitution gives the estimate.
\end{proof}

The last step deliberately weakens the sharper factor $1/(t+\|\sigma_\rho x-\sigma_\tau y\|)$ to $1/(t+\|x-y\|)$.  
This removes the chamber indices from the scalar kernel bound; 
the dependence on $(\rho,\tau)$ is then carried only by the chamber measures $\gamma_\tau$ and $\mu_\rho$, 
so the same scalar proof applies to all $|G|^2$ entries.

The chamber lifting now identifies the full operator with the finite matrix of the entries $T_{\rho\tau}$.

\begin{lemma}\label{lem:matrix-representation}
Assume $(\gamma,\mu)$ is wall-null.  For non-negative Borel $f$, for $1\le\rho\le |G|$, and for $(x,t)\in\Ccal\times(0,\infty)$,
\begin{equation}\label{eq:matrix-representation}
        (V_\mu P_\gamma f)_\rho(x,t) =  \sum_{\tau=1}^{|G|} T_{\rho\tau}(U_\gamma f)_\tau(x,t).
\end{equation}
If the scalar entries are bounded on the corresponding $L^2$ spaces, 
the same identity holds $\mu_\rho$-almost everywhere for real-valued $f\in L^2(\R^N,d\gamma)$, 
with both sides read as the difference of the positive and negative parts.
\end{lemma}

\begin{proof}
Recall that for $(x,t)\in\Ccal\times(0,\infty)$,
$$
        K_t^{\rho\tau}(x,y)=p_t(\sigma_\rho x,\sigma_\tau y), \qquad
        T_{\rho\tau}h(x,t)=\int_\Ccal K_t^{\rho\tau}(x,y)h(y)\,d\gamma_\tau(y).
$$
We first prove the identity for non-negative Borel $f$.

Using the chamber partition and the definitions of $\gamma_\tau$,
\begin{align*}
        P_\gamma f(\sigma_\rho x,t) &=
        \int_{\R^N}p_t(\sigma_\rho x,y)f(y)\,d\gamma(y)  
        =\sum_{\tau=1}^{|G|} \int_{\sigma_\tau\Ccal}   p_t(\sigma_\rho x,y)f(y)\,d\gamma(y)                 \\
        &=  \sum_{\tau=1}^{|G|}  \int_{\Ccal}  p_t(\sigma_\rho x,\sigma_\tau y)f(\sigma_\tau y)  \,d\gamma_\tau(y)                                   
        =  \sum_{\tau=1}^{|G|} T_{\rho\tau}(U_\gamma f)_\tau(x,t),
\end{align*}
where $(x,t)\in\Ccal\times(0,\infty)$.

Now let $f\in L^2(\R^N,d\gamma)$ be real-valued and write $f=f^+-f^-$.  Apply the non-negative identity to $f^+$ and $f^-$.  
In the applications where this signed identity is used, the scalar entries are already known to be bounded on $L^2$, and hence
$T_{\rho\tau}(U_\gamma f^\pm)_\tau$ is finite $\mu_\rho$-almost everywhere. Subtracting the two positive identities gives the displayed identity for $f$. 
Complex-valued functions follow by applying the same argument to real and imaginary parts.
\end{proof}

Before passing to the scalar theorem, we spell out one point about the notation.

\begin{remark}
The maps $U_\gamma$ and $V_\mu$ are liftings to a finite chamber vector. They are not projections to the orbit space.  
We do not assume that the input function or the measures are $G$-invariant; 
the different chamber masses remain visible in the entries $T_{\rho\tau}$.  
Similar chamber or orbit-adapted decompositions appear in Dunkl square-function and commutator estimates \cite{CGLWW-Square,HanLeeLiSawyerWuCalderonSubmitted}.  Here the lifting replaces orbit-mass comparison by finitely many scalar chamber tests.
\end{remark}

\section{The scalar chamber theorem}\label{sec:scalar-chamber}

Throughout this section, the wall-null pair $(\gamma,\mu)$ and the chamber component measures $\gamma_\tau,\mu_\rho$ are as in Section~\ref{sec:chamber-lifting};
we also keep the conventions $Q_\Ccal=Q\cap\Ccal$ and $\widehat Q_\Ccal=Q_\Ccal\times(0,\ell(Q)]$.

Fix one pair $(\rho,\tau)$ with $1\le\rho,\tau\le |G|$.  To lighten notation in this section, write
$$
        K_t(x,y):=K_t^{\rho\tau}(x,y),  \qquad  \nu:=\gamma_\tau,   \qquad  \eta:=\mu_\rho,
$$
and
$$
        Tf(x,t):=\int_{\Ccal}K_t(x,y)f(y)\,d\nu(y),\qquad (x,t)\in\Ccal\times(0,\infty).
$$
The adjoint is
$$
        T^*g(y):=\iint_{\Ccal\times(0,\infty)}K_t(x,y)g(x,t)\,d\eta(x,t),\qquad y\in\Ccal.
$$
The measures $\nu$ and $\eta$ are the chamber pieces of the original locally finite measures.  
Thus $\nu(Q_\Ccal)<\infty$ for every bounded dyadic cube $Q$.  
All adjoint identities are first justified in finite truncations and then passed to the monotone limit.

\begin{theorem}\label{thm:scalar-chamber}
The operator
$$
        T:L^2(\Ccal,d\nu)\to L^2(\Ccal\times(0,\infty),d\eta)
$$
is bounded if and only if the two testing conditions
\begin{equation}\label{eq:scalar-forward-test}
        \iint_{\widehat{3Q}_\Ccal}  \abs{T\one_{Q_\Ccal}(x,t)}^2\,d\eta(x,t)  \le F^2\nu(Q_\Ccal)
\end{equation}
and
\begin{equation}\label{eq:scalar-backward-test}
        \int_{3Q_\Ccal} \abs{T^*(t\one_{\widehat{Q}_\Ccal})(y)}^2\,d\nu(y)
        \le B^2 \iint_{\widehat{Q}_\Ccal}t^2\,d\eta(x,t)
\end{equation}
hold uniformly over dyadic cubes $Q$ with $Q_\Ccal\ne\varnothing$.  Here $F$ and $B$ denote the least admissible constants in these two tests.  More explicitly,
\begin{equation}\label{eq:scalar-forward-supremum}
        F^2 := \sup_{\substack{Q\in\Dcal,\ Q_\Ccal\ne\varnothing,\\
          \nu(Q_\Ccal)>0}}  \frac{  \displaystyle \iint_{\widehat{3Q}_\Ccal}  \abs{T\one_{Q_\Ccal}(x,t)}^2\,d\eta(x,t)}  {\nu(Q_\Ccal)},
\end{equation}
and
\begin{equation}\label{eq:scalar-backward-supremum}
        B^2:=\sup_{\substack{Q\in\Dcal,\ Q_\Ccal\ne\varnothing,\\
        \iint_{\widehat Q_\Ccal}t^2\,d\eta>0}} \frac{\displaystyle
        \int_{3Q_\Ccal} \abs{T^*(t\one_{\widehat{Q}_\Ccal})(y)}^2\,d\nu(y)}  {\displaystyle\iint_{\widehat Q_\Ccal}t^2\,d\eta(x,t)}.
\end{equation}
Cubes with zero denominator are omitted from the corresponding supremum.  
If there are no such cubes, the corresponding supremum is understood to be zero. Moreover,
\begin{equation}\label{eq:scalar-norm-equivalence}
        \norm{T}_{L^2(\Ccal,d\nu)\to L^2(\Ccal\times(0,\infty),d\eta)} \simeq  F+B,
\end{equation}
with constants depending only on $N$, $R$, and $\kappa$. 
\end{theorem}

The backward test is interpreted with extended non-negative integrals. 
 If $\nu\equiv0$, the entry is zero and we put $F=B=0$.  Otherwise the forward test forces the local square-finiteness needed for
$T^*(t\one_{\widehat Q_\Ccal})$ to be pointwise meaningful.  
This is the  chamber analogue of the usual two-weight $A_2$ necessity.

\begin{lemma}\label{lem:forward-implies-square-finiteness}
Assume that $\nu\not\equiv0$ and that the forward testing condition \eqref{eq:scalar-forward-test} holds with $F<\infty$.  
Then
\begin{equation}\label{eq:auto-square-local-finiteness}
        \iint_{\widehat H_\Ccal}t^2\,d\eta(x,t)<\infty
\end{equation}
for every dyadic cube $H$ with $H_\Ccal\ne\varnothing$.
\end{lemma}

\begin{proof}
Since $\nu$ is locally finite and non-zero, choose a bounded dyadic cube $Q_0$ such that
$$
        0<\nu(Q_{0,\Ccal})<\infty .
$$
For the given cube $H$, take a dyadic cube $Q$ with
$$
        Q_0\cup H\subseteq Q,  \qquad   \ell(H)\le\ell(Q).
$$
Then
$$
        \widehat H_\Ccal\subseteq\widehat{3Q}_\Ccal,   \qquad  0<\nu(Q_\Ccal)<\infty .
$$

For $x\in H_\Ccal$, $0<t\le\ell(H)$, and $y\in Q_\Ccal$, all variables stay in a bounded set depending only on $Q$.  In \eqref{eq:poisson-kernel-formula},
$$
        \xi\in\co(O(\sigma_\rho x))    \quad\Longrightarrow\quad \norm{\xi}\le\norm{x},
$$
and so
$$
        t^2+A(\sigma_\rho x,\sigma_\tau y,\xi)^2\le C_Q.
$$
Hence $K_t(x,y)\ge c_Qt$ on $\widehat H_\Ccal\times Q_\Ccal$, and
$$
        T\one_{Q_\Ccal}(x,t)   \ge c_Qt\,\nu(Q_\Ccal),   \qquad (x,t)\in\widehat H_\Ccal.
$$
The forward test for $Q$ gives
$$
        c_Q^2\nu(Q_\Ccal)^2  \iint_{\widehat H_\Ccal}t^2\,d\eta
        \le \iint_{\widehat{3Q}_\Ccal}|T\one_{Q_\Ccal}|^2\,d\eta  \le F^2\nu(Q_\Ccal)<\infty,
$$
which proves \eqref{eq:auto-square-local-finiteness}.
\end{proof}

The necessity of the testing conditions now follows without an additional hypothesis.  
Applying the operator bound to $\one_{Q_\Ccal}$ gives the forward test.

If $\nu\equiv0$, the backward test is trivial.  
Otherwise the forward test just obtained and Lemma~\ref{lem:forward-implies-square-finiteness} give
$$
        t\one_{\widehat Q_\Ccal}\in L^2(\Ccal\times(0,\infty),d\eta).
$$
Applying the dual operator bound to this function gives \eqref{eq:scalar-backward-test}.  Thus
$$
        F+B\lesssim
        \norm{T}_{L^2(\Ccal,d\nu)\to L^2(\Ccal\times(0,\infty),d\eta)}.
$$
The rest of this section proves sufficiency.

\subsection{Two kernel comparisons}

The first comparison is the maximum principle used in the stopping-time proof.

\begin{lemma}\label{lem:kernel-maximum-comparison}
There is a constant $C_*$ with the following property.  
Let $Q$ be a dyadic cube with $Q_\Ccal\ne\varnothing$, let $y\in Q_\Ccal$, and let $z\in 9Q_\Ccal$.  
If $(x,t)\notin\widehat{3Q}_\Ccal$, then
$$
        K_t(x,y)\le C_*K_t(x,z).
$$
The constant is independent of $Q,x,y,z,t$.
\end{lemma}

\begin{proof}
Put
$$
        u:=\sigma_\rho x,   \qquad  v:=\sigma_\tau y, \qquad   w:=\sigma_\tau z.
$$
We compare the two kernels by using the integral representation with the same representing measure $\lambda_u$:
$$
        K_t(x,y)=p_t(u,v),   \qquad   K_t(x,z)=p_t(u,w).
$$
For $\xi\in\co(O(u))$ we have $\norm{\xi}\le\norm{u}$, because $\co(O(u))$ is the convex hull of points all having norm $\norm{u}$.  Hence
$$
        A(u,a,\xi)^2   =\norm{a-\xi}^2+\norm{u}^2-\norm{\xi}^2,    \qquad a\in\R^N,
$$
and the map $a\mapsto A(u,a,\xi)$ is $1$-Lipschitz.  Therefore
\begin{equation}\label{eq:A-Lipschitz-yz}
        A(u,w,\xi)   \le  A(u,v,\xi)+\norm{v-w}.
\end{equation}
Since $y\in Q$ and $z\in9Q$, orthogonality of $\sigma_\tau$ gives
\begin{equation}\label{eq:yz-size}
        \norm{v-w}=\norm{y-z}\le C_N\ell(Q).
\end{equation}

First assume $x\notin3Q$.  Since $y\in Q$, we have
$$
        \norm{x-y}\ge c_N\ell(Q).
$$
By the chamber identity,
$$
        d(u,v)=d(\sigma_\rho x,\sigma_\tau y)=\norm{x-y}\ge c_N\ell(Q).
$$
Applying \eqref{eq:A-lower-d} with $y_0=u$ and $x_0=v$ gives
$$
        A(u,v,\xi)\ge d(u,v)\ge c_N\ell(Q),   \qquad \xi\in\co(O(u)).
$$
Together with \eqref{eq:A-Lipschitz-yz} and \eqref{eq:yz-size}, this gives
$$
        A(u,w,\xi)  \le C A(u,v,\xi).
$$
Consequently
$$
        t^2+A(u,w,\xi)^2   \le C\bigl(t^2+A(u,v,\xi)^2\bigr).
$$
Substituting this estimate into the integral formula \eqref{eq:poisson-kernel-formula} gives  $K_t(x,y)\le C K_t(x,z)$.

It remains to consider the case $x\in3Q$ and $(x,t)\notin\widehat{3Q}_\Ccal$.  
Then necessarily $t>\ell(Q)$.  By \eqref{eq:A-Lipschitz-yz} and \eqref{eq:yz-size},
$$
        A(u,w,\xi) \le A(u,v,\xi)+C_N\ell(Q)  \le A(u,v,\xi)+C_Nt.
$$
Thus
$$
        t^2+A(u,w,\xi)^2  \le C\bigl(t^2+A(u,v,\xi)^2\bigr),
$$
and the same substitution into \eqref{eq:poisson-kernel-formula} proves the comparison.
\end{proof}

The second comparison is the local vertical estimate used in the stopping-time part.  
The reference point must lie in the chamber.  A dyadic cube may cross a wall, and then its Euclidean centre may lie outside $\Ccal$.  
We therefore use an arbitrary point of $Q_\Ccal$ instead of the centre.

\begin{lemma}\label{lem:vertical-comparison}
Let $Q$ be a dyadic cube with $Q_\Ccal\ne\varnothing$.  Choose once and for all a point
$
        x_Q^\Ccal\in Q_\Ccal.
$
Let $E\subseteq\Ccal$ be a Borel set satisfying
$
        E\cap 3Q_\Ccal=\varnothing.
$
If $(x,t)\in\widehat{Q}_\Ccal$, then
\begin{equation}\label{eq:vertical-comparison}
        T\one_E(x,t)   \simeq   \frac{t}{\ell(Q)} T\one_E(x_Q^\Ccal,\ell(Q)).
\end{equation}
The constants are independent of $E,Q,x,t$ and of the chosen point $x_Q^\Ccal$.
\end{lemma}

\begin{proof}
Let $y\in E$ and $\sigma_\rho,\sigma_\tau\in G$.  Since
$$
        x\in Q_\Ccal,  \qquad   y\notin 3Q_\Ccal,   \qquad    t\le\ell(Q),
$$
we have
$$
        \norm{x-y}\gtrsim \ell(Q),   \qquad  d(\sigma_\rho x,\sigma_\tau y)=\norm{x-y}\gtrsim\ell(Q),
$$
where the second identity is \eqref{eq:chamber-orbit-identity}.

Use the symmetry \eqref{eq:p-symmetry} of the kernel in the form
$$
        K_t(x,y)=p_t(\sigma_\rho x,\sigma_\tau y)    =p_t(\sigma_\tau y,\sigma_\rho x).
$$
Both kernels now use the same representing measure  $\lambda_{\sigma_\tau y}$ in \eqref{eq:poisson-kernel-formula}.

Applying \eqref{eq:A-lower-d} with first variable $\sigma_\tau y$ and second  variable $\sigma_\rho x$ gives, 
for every $\xi\in\co(O(\sigma_\tau y))$,
$$
        A(\sigma_\tau y,\sigma_\rho x,\xi)  \gtrsim \ell(Q).
$$
The same argument with $x_Q^\Ccal$ in place of $x$ gives
$$
        A(\sigma_\tau y,\sigma_\rho x_Q^\Ccal,\xi)   \gtrsim \ell(Q),  \qquad \xi\in\co(O(\sigma_\tau y)).
$$
Indeed, $x_Q^\Ccal\in Q_\Ccal$ and $y\notin3Q_\Ccal$, so $\|x_Q^\Ccal-y\|\gtrsim\ell(Q)$.  
Since both $x$ and $x_Q^\Ccal$ lie in $Q$, we have
$$
        \norm{x-x_Q^\Ccal}\lesssim \ell(Q).
$$
Thus the triangle inequality gives both estimates
$$
        A(\sigma_\tau y,\sigma_\rho x_Q^\Ccal,\xi)  \le
        A(\sigma_\tau y,\sigma_\rho x,\xi)+C\ell(Q)  \lesssim A(\sigma_\tau y,\sigma_\rho x,\xi)
$$
and
$$
        A(\sigma_\tau y,\sigma_\rho x,\xi)   \le
        A(\sigma_\tau y,\sigma_\rho x_Q^\Ccal,\xi)+C\ell(Q)  \lesssim  A(\sigma_\tau y,\sigma_\rho x_Q^\Ccal,\xi).
$$
Therefore
$$
        t^2+A(\sigma_\tau y,\sigma_\rho x,\xi)^2  \simeq   \ell(Q)^2+A(\sigma_\tau y,\sigma_\rho x_Q^\Ccal,\xi)^2.
$$
Using the Poisson kernel formula and the fact that the numerator is linear in the height variable, we obtain the pointwise comparison
$$
        K_t(x,y)  \simeq  \frac{t}{\ell(Q)}K_{\ell(Q)}(x_Q^\Ccal,y).
$$
Integrating this estimate over $y\in E$ proves \eqref{eq:vertical-comparison}.
\end{proof}

\subsection{Whitney decomposition inside the chamber}

We use ordinary dyadic cubes in $\R^N$, but all sets are intersected with $\Ccal$.

\begin{lemma}\label{lem:whitney-chamber}
Let $\Omega\subseteq\Ccal$ be relatively open and not equal to $\Ccal$.  There is a collection $\mathcal I(\Omega)\subseteq\Dcal$ of maximal dyadic cubes such that
\begin{equation}\label{eq:whitney-def}
        3Q_\Ccal\subseteq\Omega,  \qquad   9Q_\Ccal\cap(\Ccal\setminus\Omega)\ne\varnothing.
\end{equation}
Moreover:
\begin{enumerate}
\item $\Omega=\bigsqcup_{Q\in\mathcal I(\Omega)}Q_\Ccal$ with the half-open dyadic convention;

\smallskip

\item the cubes $Q$ in $\mathcal I(\Omega)$ are pairwise disjoint;

\smallskip
\item for every fixed $a>1$,
$$
        \sum_{Q\in\mathcal I(\Omega)}\one_{aQ_\Ccal}  \le C_{a,N}
$$
pointwise on $\Ccal$; in particular,
$$
        \sum_{Q\in\mathcal I(\Omega)}\one_{3Q_\Ccal}    \le C_N\one_\Omega;
$$

\item if $\Omega$ and $\Omega'$ belong to a nested family of relatively open subsets of $\Ccal$,  and if $Q\in\mathcal I(\Omega)$ and $Q'\in\mathcal I(\Omega')$ with $Q\subsetneq Q'$, then $\Omega\subseteq\Omega'$.
\end{enumerate}
\end{lemma}

\begin{proof}
For each $x\in\Omega$ choose a small dyadic cube $Q$ containing $x$ with $3Q_\Ccal\subseteq\Omega$.   
This is possible because $\Omega$ is relatively open in $\Ccal$.  
Enlarge $Q$ dyadically while the same condition remains true. 
The enlargement stops: since $\Omega\ne\Ccal$, large ancestors eventually meet $\Ccal\setminus\Omega$ after dilation.

The maximal cubes obtained in this way satisfy \eqref{eq:whitney-def}.  
With the half-open dyadic convention they cover $\Omega$ exactly, and two distinct maximal cubes are disjoint.

The bounded overlap is the standard dyadic Whitney overlap estimate.  If $aQ$ and $aQ'$ meet and $Q,Q'\in\mathcal I(\Omega)$, then
$$
        \ell(Q)\simeq_{a,N}\ell(Q').
$$
Indeed, if two cubes $Q,Q'\in\mathcal W(\Omega)$ have intersecting fixed dilates and $\ell(Q)$ is much smaller than $\ell(Q')$, then a dyadic ancestor of $Q$
would still be contained in $\Omega$ and would still satisfy the Whitney condition, contradicting the maximality of $Q$. 
Hence all Whitney cubes whose fixed dilates meet have comparable sidelengths. 
Since the Whitney cubes themselves are pairwise disjoint, these fixed dilates have uniformly bounded overlap; 
equivalently, at most a uniformly bounded number of them can contain  any fixed point.

For the last assertion, the nesting assumption gives two alternatives:
$$
        \Omega\subseteq\Omega'  \quad\text{or}\quad   \Omega'\subseteq\Omega.
$$
The second is impossible.  If $\Omega'\subseteq\Omega$, then $3Q'_\Ccal\subseteq\Omega$, 
contradicting the maximality of $Q$ in $\mathcal I(\Omega)$ because $Q\subsetneq Q'$.  
Hence $\Omega\subseteq\Omega'$.
\end{proof}

We need the following continuity property of $T^*$.

\begin{lemma}\label{lem:level-sets-proper}
Let $g:\Ccal\times(0,\infty)\to[0,\infty)$ be bounded and supported on a compact subset of $\Ccal\times(0,\infty)$; equivalently,
$$
        \supp\,g\subseteq K\times[a,b], \qquad \dist(K,\partial\Ccal)>0,   \qquad 0<a<b<\infty,
$$
for some compact $K\subseteq\Ccal$.  Then $T^*g$ is finite and continuous on $\Ccal$, and
$$
        T^*g(y)\longrightarrow0, \qquad \norm{y}\to\infty,\quad y\in\Ccal.
$$
Consequently every level set 
\begin{equation}\label{eq:level-sets-Omega}
     \Omega_k=\{y\in\Ccal:\,T^*g(y)>2^k\},
\end{equation}
with $k\in\Z$ is relatively open, bounded, and proper.
\end{lemma}

\begin{proof}
Use the compact support description in the statement and write $\supp\,g\subseteq K\times[a,b]$.

For $(x,t)\in K\times[a,b]$, the kernel $K_t(x,y)$ is continuous in $y$, 
and the denominator in the integral representation is bounded from below by $a^{\mathbf N+1}$.  
Since $g$ is bounded and $\eta(K\times[a,b])<\infty$, dominated convergence gives continuity and finiteness of $T^*g$ on $\Ccal$.

For $y\in\Ccal$ and $(x,t)\in K\times[a,b]$, the chamber identity gives
$$
      d(\sigma_\rho x,\sigma_\tau y)=\norm{x-y}\to\infty,\qquad \norm{y}\to\infty.
$$
The kernel upper bound  \eqref{eq:matrix-kernel-upper} therefore yields
$$
        \sup_{(x,t)\in K\times[a,b]}K_t(x,y)\longrightarrow0.
$$
Multiplying by the bounded function $g$ and integrating over the finite measure set $K\times[a,b]$ proves the decay.  
The asserted properties of $\Omega_k$ follow immediately.
\end{proof}

Let $\Omega_k$ with $k\in\Z$ be as in Lemma~\ref{lem:level-sets-proper}. For a family $\mathcal I_k:=\mathcal I(\Omega_k)$, define
$$
        \widehat\Omega_k   :=  \bigcup_{Q\in\mathcal I_k}\widehat{Q}_\Ccal.
$$
We use the following elementary counting fact.  It is stated here, where it is first needed, 
and it will be cited again in Appendix~\ref{sec:dyadic-packing}; the proof is not repeated there.

\begin{lemma}\label{lem:level-box-overlap}
Fix $m\ge1$.  Let $\Omega_k$ be as in \eqref{eq:level-sets-Omega}, and $\mathcal I_k:=\mathcal I(\Omega_k)$, $k\in\Z$. Then
\begin{equation}\label{eq:level-box-overlap}
\sum_{k\in\Z}\sum_{Q\in\mathcal I_k}
        \one_{\widehat{3Q}_\Ccal\setminus\widehat\Omega_{k+m+1}}(x,t) \le C_{N,m}, \qquad (x,t)\in\Ccal\times(0,\infty).
\end{equation}
\end{lemma}

\begin{proof}
Fix $(x,t)\in\Ccal\times(0,\infty)$.  For a fixed level $k$, Lemma~\ref{lem:whitney-chamber} gives
$$
        \sum_{Q\in\mathcal I_k}\one_{3Q_\Ccal}(x)\le C_N,
$$
so only boundedly many cubes $Q$ can contribute at that level.

It remains to count the levels.  If
$$
        (x,t)\in\widehat{3Q}_\Ccal\setminus\widehat\Omega_{k+m+1},
$$
then $x\in\Omega_k$ and $t\le\ell(Q)$.  If $x\notin\Omega_{k+m+1}$, nesting leaves at most $m+1$ possible values of $k$.

If $x\in\Omega_{k+m+1}$, let $J\in\mathcal I_{k+m+1}$ be the Whitney cube containing $x$.  
The condition $(x,t)\notin\widehat\Omega_{k+m+1}$ gives $t>\ell(J)$.  
Hence the pair $(Q,J)$ gives a crossing of the height $t$ between the levels
$k$ and $k+m+1$.  Whitney maximality and bounded neighbour overlap allow only
$O_{N,m}(1)$ such crossings.
This proves \eqref{eq:level-box-overlap}.
\end{proof}
\subsection{Proof of the sufficiency of scalar chamber theorem}

If $\nu\equiv0$, then $L^2(\Ccal,d\nu)=\{0\}$ and the scalar estimate is trivial.  Henceforth
$$
        \nu\not\equiv0.
$$
By Lemma~\ref{lem:forward-implies-square-finiteness}, every chamber box used below has finite $t^2\,d\eta$-mass.

We prove the dual estimate
\begin{equation}\label{eq:dual-estimate-target}
        \norm{T^*g}_{L^2(\Ccal,d\nu)} \lesssim   (F+B)\norm{g}_{L^2(\Ccal\times(0,\infty),d\eta)}.
\end{equation}
By duality this is equivalent to the desired bound for $T$.  
It suffices to prove this estimate for $g$ non-negative, bounded, and compactly supported in $\Ccal\times(0,\infty)$.  
The general case follows by replacing $g$ with $|g|$, approximating from below, and then using monotone convergence and duality.

The stopping-time sums are also read in a finite form.  
We insert $\one_{Q_{0,\Ccal}}$ in the outer $\nu$-integral, restrict $k$ to a finite interval, 
and use only finitely many side lengths.  The constants below do not see these cutoffs.  
After the estimates are obtained, $Q_0\uparrow\R^N$ and the level and scale cutoffs are removed.  
The notation below suppresses these cutoffs.

Lemma~\ref{lem:level-sets-proper} makes the level sets relatively open, bounded, and proper.  
Thus Lemma~\ref{lem:whitney-chamber} applies. 
Let $\Omega_k$ be given by \eqref{eq:level-sets-Omega}, let $\mathcal I_k=\mathcal I(\Omega_k)$, 
and choose a positive integer $m$ to be fixed later.  Set
$$
        \mathbf{C}_k:=\Omega_{k+m}\setminus\Omega_{k+m+1}.
$$
The sets $\mathbf{C}_k$ are disjoint in $k$, and they cover $\{y\in\Ccal:T^*g(y)>0\}$.  
The complement does not contribute to the $L^2(\Ccal,d\nu)$ norm of $T^*g$.  
If $y\in\mathbf{C}_{k}$, then $T^*g(y)\le 2^{k+m+1}$, and hence
\begin{align}\label{eq:level-decomp-first}
        \int_\Ccal \abs{T^*g(y)}^2\,d\nu(y)
        &\lesssim_m  \sum_k 2^{2k}\nu(\Omega_k\cap \mathbf{C}_k)     \le  C_m\sum_k2^{2k}  \sum_{Q\in\mathcal I_k}\nu(F_k(Q)),
\end{align}
where
$$
        F_k(Q):=Q_\Ccal\cap \mathbf{C}_k.
$$
Fix $0<\delta<1$.  Decompose the cubes in $\mathcal I_k$ into three classes:
\begin{align*}
        \Gamma_k^1 &:=  \Bigl\{Q\in\mathcal I_k:\nu(F_k(Q))\le \delta\nu(Q_\Ccal)\Bigr\},       \\
        \Gamma_k^2   &:=  \Bigl\{Q\in\mathcal I_k:\nu(F_k(Q))> \delta\nu(Q_\Ccal),\     \ \alpha_k(Q)>\beta_k(Q)\Bigr\},                          \\
        \Gamma_k^3 &:=  \Bigl\{Q\in\mathcal I_k:\nu(F_k(Q))> \delta\nu(Q_\Ccal),\    \ \alpha_k(Q)\le\beta_k(Q)\Bigr\},
\end{align*}
where
\begin{align*}
        \alpha_k(Q) &:= \iint_{\widehat{3Q}_\Ccal\setminus\widehat\Omega_{k+m+1}}
        T\one_{F_k(Q)}(x,t)g(x,t)\,d\eta(x,t),     \\
        \beta_k(Q)  &:=  \iint_{\widehat{3Q}_\Ccal\cap\widehat\Omega_{k+m+1}}  T\one_{F_k(Q)}(x,t)g(x,t)\,d\eta(x,t).
\end{align*}
Correspondingly write the right side of \eqref{eq:level-decomp-first} as $L_1+L_2+L_3$ based on $\Gamma_k^i$, $i=1,2,3$.

\subsubsection*{The small part}
For $L_1$, the bounded overlap of Whitney cubes gives
\begin{align*}
        L_1 &\le  C_m\sum_k2^{2k}  \sum_{Q\in\Gamma_k^1}\nu(F_k(Q))                    \le
        C_m\delta   \sum_k2^{2k}     \sum_{Q\in\mathcal I_k}\nu(Q_\Ccal)                \\
        &\le   C_m\delta   \sum_k2^{2k}\nu(\Omega_k)                          
        \le   C_m\delta \int_\Ccal \abs{T^*g(y)}^2\,d\nu(y).
\end{align*}
At the finite truncation stage fixed above, the last integral is finite by construction.
Choose $\delta>0$ sufficiently small, depending only on $m$ and the Whitney overlap constants, 
so that this term can be absorbed into the left side.
Once the estimate is obtained with constants independent of the truncation, monotone convergence removes the truncations.

\subsubsection*{The part controlled by the forward test}
The local lower bound is the basic estimate here.  By Lemma~\ref{lem:whitney-chamber}, for every $Q\in\mathcal I_k$ there is a point
$$
        z_Q\in9Q_\Ccal\cap(\Ccal\setminus\Omega_k).
$$
Lemma~\ref{lem:kernel-maximum-comparison} gives, for $y\in Q_\Ccal$,
$$
        T^*(g\one_{(\widehat{3Q}_\Ccal)^c})(y) \le  C_*T^*g(z_Q) \le   C_*2^k.
$$
Choose $m$ so large that $2^m>C_*+2$.  If $y\in F_k(Q)$, then $y\in\Omega_{k+m}$, and hence
\begin{equation}\label{eq:local-lower-bound}
        T^*(g\one_{\widehat{3Q}_\Ccal})(y)  \ge  2^{k+m}-C_*2^k  \ge 2^k.
\end{equation}
Integrating \eqref{eq:local-lower-bound} over $F_k(Q)$ gives
\begin{equation}\label{eq:alpha-beta-lower}
        2^k\nu(F_k(Q))  \le  \alpha_k(Q)+\beta_k(Q).
\end{equation}

If $Q\in\Gamma_k^2$, then $\alpha_k(Q)>\beta_k(Q)$, so
$$
        2^k\nu(F_k(Q))\le 2\alpha_k(Q).
$$
Therefore
$$
        L_2  \lesssim_{m}  \sum_k\sum_{Q\in\Gamma_k^2}   \frac{\alpha_k(Q)^2}{\nu(F_k(Q))}.
$$
By Cauchy's inequality, positivity, and the forward test applied to $Q$,
\begin{align*}
        \alpha_k(Q)^2 &\le  \left(  \iint_{\widehat{3Q}_\Ccal}  \abs{T\one_{F_k(Q)}(x,t)}^2\,d\eta(x,t) \right)
        \left( \iint_{\widehat{3Q}_\Ccal\setminus\widehat\Omega_{k+m+1}}    \abs{g(x,t)}^2\,d\eta(x,t) \right)                           \\
        &\le   F^2\nu(Q_\Ccal)  \iint_{\widehat{3Q}_\Ccal\setminus\widehat\Omega_{k+m+1}}  \abs{g(x,t)}^2\,d\eta(x,t).
\end{align*}
Since $\nu(F_k(Q))>\delta\nu(Q_\Ccal)$, for $Q\in{\Gamma_k^2}$,
\begin{align*}
        L_2  &\lesssim_{m}   \delta^{-1}F^2   \iint_{\Ccal\times(0,\infty)}  \sum_k\sum_{Q\in\Gamma_k^2}
        \one_{\widehat{3Q}_\Ccal\setminus\widehat\Omega_{k+m+1}}(x,t)  \abs{g(x,t)}^2\,d\eta(x,t)                    \\
        &\lesssim_{N,m}   \delta^{-1}F^2\norm{g}_{L^2(\Ccal\times(0,\infty),d\eta)}^2.
\end{align*}
The last inequality follows from the level-box overlap estimate \eqref{eq:level-box-overlap}.

\subsubsection*{The part controlled by the backward test and the principal cubes}
If $Q\in\Gamma_k^3$, then \eqref{eq:alpha-beta-lower} gives
$$
        2^k\nu(F_k(Q))\le2\beta_k(Q).
$$
Thus
$$
        L_3  \lesssim_{m} \sum_k\sum_{Q\in\Gamma_k^3}  \frac{\beta_k(Q)^2}{\nu(F_k(Q))}
        \lesssim_{m}  \delta^{-1}   \sum_k\sum_{Q\in\Gamma_k^3}  \frac{\beta_k(Q)^2}{\nu(Q_\Ccal)}.
$$

The upper part of the box is decomposed next.  Near a wall the centre of a Whitney cube is not a reliable chamber point, 
so centers are avoided.  Nor do we ask a Whitney cube itself to sit inside $3Q$.  
The right ambient object is the following family of maximal dyadic subcubes of $3Q$.

For $Q\in\Gamma_k^3$, let $\mathcal R_{k+m+1}(Q)$ be the collection of maximal dyadic cubes $H$ such that
$$
        H\subseteq 3Q,  \qquad  H_\Ccal\ne\varnothing,  \qquad
        \ell(H)\le\ell(Q),  \qquad    3H_\Ccal\subseteq\Omega_{k+m+1}.
$$
The cubes in $\mathcal R_{k+m+1}(Q)$ are pairwise disjoint by maximality. The use of ambient dyadic cubes here is deliberate: 
the chamber is intersected only after the dyadic containment relations have been fixed.

They cover the relevant part of the upper box:
\begin{equation}\label{eq:RkQ-cover}
        \widehat{3Q}_\Ccal\cap\widehat\Omega_{k+m+1}
        \subseteq \bigcup_{H\in\mathcal R_{k+m+1}(Q)}\widehat H_\Ccal.
\end{equation}
Indeed, take $(x,t)$ in the left-hand side.  Then
$$
        x\in3Q_\Ccal,   \qquad   0<t\le\ell(Q).
$$
Choose $J\in\mathcal I_{k+m+1}$ with
$$
        x\in J_\Ccal,   \qquad   t\le\ell(J).
$$
Write
$$
        3Q=\bigsqcup_{\theta=1}^{3^N} I_\theta
$$
with $\ell(I_\theta)=\ell(Q)$, and let $I_\theta$ be the piece containing $x$.  There are two cases.

If $\ell(J)\le\ell(Q)$, then $J$ and $I_\theta$ meet and $\ell(J)\le\ell(I_\theta)$, so
$$
        J\subseteq I_\theta\subseteq3Q.
$$
Set $H_0=J$.  If $\ell(J)>\ell(Q)$, then $I_\theta\subseteq J$; set $H_0=I_\theta$.  In both cases
$$
        H_0\subseteq3Q,  \qquad   t\le\ell(H_0),   \qquad   3H_{0,\Ccal}\subseteq3J_\Ccal\subseteq\Omega_{k+m+1}.
$$
For $H_0=I_\theta$ we used only that the triple of a dyadic subcube is contained in the triple of its dyadic ancestor.  
Maximality now gives a cube $H\in\mathcal R_{k+m+1}(Q)$ with $H_0\subseteq H$, hence $(x,t)\in\widehat H_\Ccal$.

Also, for every $H\in\mathcal R_{k+m+1}(Q)$,
\begin{equation}\label{eq:H-separated-from-F}
        3H_\Ccal\cap F_k(Q)=\varnothing,
\end{equation}
because
$$
        3H_\Ccal\subseteq\Omega_{k+m+1},  \qquad
        F_k(Q)\subseteq \mathbf{C}_k\subseteq\Ccal\setminus\Omega_{k+m+1}.
$$
Thus Lemma~\ref{lem:vertical-comparison} applies to $E=F_k(Q)$ and $J=H$.

Put
\begin{equation}\label{eq:tildemu}
        d\widetilde\eta(x,t):=t^2\,d\eta(x,t).
\end{equation}
By Lemma~\ref{lem:forward-implies-square-finiteness}, $\widetilde\eta(\widehat H_\Ccal)<\infty$ for every dyadic cube $H$ with
$H_\Ccal\ne\varnothing$.  For such an $H$, define
$$
        \Pi(H):=
        \begin{cases}
        \displaystyle  \frac1{\widetilde\eta(\widehat H_\Ccal)}  \iint_{\widehat H_\Ccal}\frac{g(x,t)}{t}\,d\widetilde\eta(x,t),
        & \widetilde\eta(\widehat H_\Ccal)>0,\\[12pt]
          0, & \widetilde\eta(\widehat H_\Ccal)=0.
        \end{cases}
$$
This convention avoids division by zero.  
Terms with $\widetilde\eta(\widehat H_\Ccal)=0$ do not contribute to any of the sums below.

Using \eqref{eq:RkQ-cover}, positivity, \eqref{eq:H-separated-from-F}, and Lemma~\ref{lem:vertical-comparison}, we get
\begin{align}
        \beta_k(Q)
        &\le  \sum_{H\in\mathcal R_{k+m+1}(Q)}   \iint_{\widehat H_\Ccal}   T\one_{F_k(Q)}(x,t)g(x,t)\,d\eta(x,t)            \nonumber\\
        &\lesssim   \sum_{H\in\mathcal R_{k+m+1}(Q)}  \Pi(H)  \iint_{\widehat H_\Ccal}
        T\one_{Q_\Ccal}(x,t)\frac{d\widetilde\eta(x,t)}{t}.
        \label{eq:beta-Pi}
\end{align}
For the last inequality, Lemma~\ref{lem:vertical-comparison} gives
$$
        T\one_{F_k(Q)}(x,t) \simeq   \frac{t}{\ell(H)}T\one_{F_k(Q)}(x_H^\Ccal,\ell(H)),   \qquad (x,t)\in\widehat H_\Ccal.
$$
The reverse comparison is
$$
        \frac{t}{\ell(H)}T\one_{F_k(Q)}(x_H^\Ccal,\ell(H))  \lesssim
        T\one_{F_k(Q)}(x,t),  \qquad (x,t)\in\widehat H_\Ccal.
$$
The left-hand reference value is independent of $(x,t)$.  
Multiplying this pointwise inequality by $t^{-1}d\widetilde\eta(x,t)$ and integrating over $\widehat H_\Ccal$ gives
$$
        T\one_{F_k(Q)}(x_H^\Ccal,\ell(H)) \lesssim  \frac{\ell(H)}{\widetilde\eta(\widehat H_\Ccal)}  
        \iint_{\widehat H_\Ccal}  T\one_{F_k(Q)}(x,t)\frac{d\widetilde\eta(x,t)}{t}.
$$
Substituting this into the first comparison and using $F_k(Q)\subseteq Q_\Ccal$ and positivity of $T$ gives \eqref{eq:beta-Pi}.

We pass to principal cubes.  In the finite truncation fixed at the start of the proof, 
choose a large dyadic cube $Q_0$ containing every cube that appears in the sums.  The estimates below do not depend on this choice.

Define the stopping family $\Gcal$ recursively.  Put $Q_0\in\Gcal$.  
If $S\in\Gcal$, its children are the maximal dyadic subcubes $H\subsetneq S$ such that
$$
        \Pi(H)>10\Pi(S),
$$
then $H\in\Gcal$.  If $\Pi(S)=0$, there are no such children: 
since $g/t\ge0$, the equality $\Pi(S)=0$ implies $g/t=0$ $\widetilde\eta$-a.e.\ on $\widehat S_\Ccal$, 
and hence every subcube has average zero.

For every dyadic cube $H\subseteq Q_0$, let
$$
        \pi H:=\min\left\{S\in\Gcal:H\subseteq S\right\}
$$
be the smallest principal cube containing $H$.

The principal cubes satisfy the Carleson packing estimate
\begin{equation}\label{eq:principal-carleson}
        \sum_{S'\in\Gcal:S'\subseteq S} \widetilde\eta(\widehat{S'}_\Ccal)
        \lesssim   \widetilde\eta(\widehat S_\Ccal).
\end{equation}
Indeed, if $S'$ ranges over the children of $S$ and $\Pi(S)>0$, then the children are disjoint and
$$
        10\Pi(S)\widetilde\eta(\widehat{S'}_\Ccal)  <    \iint_{\widehat{S'}_\Ccal}\frac{g}{t}\,d\widetilde\eta.
$$
Summing over the children gives
$$
        \sum_{S'\in\operatorname{ch}_{\Gcal}(S)}\widetilde\eta(\widehat{S'}_\Ccal)
         \le \frac1{10\Pi(S)} \iint_{\widehat S_\Ccal}\frac{g}{t}\,d\widetilde\eta
        =  \frac1{10}\widetilde\eta(\widehat S_\Ccal),
$$
where $\operatorname{ch}_{\Gcal}(S)$ denotes the stopping children of $S$. 
 If $\Pi(S)=0$, there are no children, as observed above.  
 Iterating over the generations gives \eqref{eq:principal-carleson}.

The dyadic maximal operator relative to $\widetilde\eta$ (see \eqref{eq:tildemu}) and chamber boxes is
$$
        M_{\widetilde\eta}h(x,t)
        := \sup_{H\in{\Dcal}:\,(x,t)\in\widehat H_\Ccal} \frac1{\widetilde\eta(\widehat H_\Ccal)}
        \iint_{\widehat H_\Ccal}\abs{h}\,d\widetilde\eta,
$$
where cubes with $\widetilde\eta(\widehat H_\Ccal)=0$ are ignored.  
The embedding used here is a tree statement on the filtered measure space $(\Ccal\times(0,\infty),\widetilde\eta)$ 
whose atoms are the half-open chamber boxes $\widehat H_\Ccal$.  
It uses only the Carleson packing condition \eqref{eq:principal-carleson} for this dyadic tree.  
It does not use doubling of $\widetilde\eta$ and does not use the Euclidean Hardy--Littlewood maximal operator.

For the maximal estimate, the boxes in this half-open filtration are nested or disjoint.  
Thus $M_{\widetilde\eta}$ is dominated by the dyadic martingale maximal operator for the filtration generated by these boxes.  
Doob's maximal inequality gives its weak $(1,1)$ and $L^2(d\widetilde\eta)$ bounds with absolute constants.  
Applying the tree Carleson embedding to $h=g/t$ and then using Doob's inequality gives
\begin{equation}\label{eq:carleson-embedding-Pi}
        \sum_{S\in\Gcal} \Pi(S)^2\,\widetilde\eta(\widehat S_\Ccal)
        \lesssim  \iint_{\Ccal\times(0,\infty)} \left(M_{\widetilde\eta}(g/t)\right)^2\,d\widetilde\eta
        \lesssim   \norm{g}_{L^2(\Ccal\times(0,\infty),d\eta)}^2.
\end{equation}

The estimates of $L_{3,1}$ and $L_{3,2}$ require two finite packing facts. We use the restricted principal-cube count from
Sawyer--Wheeden~\cite[p.~861]{SawyerWheeden}; 
Lacey's primer gives the same stopping-time mechanism in modern notation~\cite{LaceyPrimer}.
Appendix~\ref{sec:dyadic-packing} gives the finite chamber version. Intersecting boxes with $\Ccal$ changes their measures, 
not the dyadic tree or the principal-cube chain count.

\begin{remark}\label{rem:heavy-not-carleson}
The heavy condition
$$
        \nu(F_k(Q))>\delta\nu(Q_\Ccal)
$$
controls only the multiplicity with which a fixed Whitney neighbour can occur.  
It is not a Carleson packing condition for the independent measure $d\widetilde\eta=t^2\,d\eta$.  
The $\widetilde\eta$-packing belowcomes from the principal-cube stopping construction.
\end{remark}

\begin{lemma}\label{lem:principal-neighbour-packing-L31}
With the notation above,
\begin{equation}\label{eq:principal-neighbour-packing-L31}
        \sum_{k,Q,\theta} \Pi(\pi I_\theta)^2\, \widetilde\eta(\widehat{I_\theta}_\Ccal)
        \le  C_{N,\delta}  \sum_{S\in\Gcal} \Pi(S)^2\widetilde\eta(\widehat S_\Ccal),
\end{equation}
where the sum on the left is restricted to triples with $1\le\theta\le3^N$, $k\in\Z$, and $Q\in\Gamma_k^3$.  
The same convention is used below.
\end{lemma}

\begin{proof}
Apply the first packing estimate in Proposition~\ref{prop:app-dyadic-sawyer-inputs} with $\Lambda=\widetilde\eta$.  
The proposition is stated in the chamber-box notation, so it gives exactly \eqref{eq:principal-neighbour-packing-L31}.
\end{proof}

\begin{lemma}\label{lem:principal-packing-L32}
With the notation above,
\begin{equation}\label{eq:principal-packing-L32}
\sum_{k,Q,\theta}
        \sum_{\substack{H\in\mathcal R_{k+m+1}(Q):\ H\subseteq I_\theta,\\
                        \pi H\subsetneq\pi I_\theta}} \widetilde\eta(\widehat H_\Ccal)\Pi(\pi H)^2                         
        \le  C_{N,m}   \sum_{S\in\Gcal}\Pi(S)^2\widetilde\eta(\widehat S_\Ccal).
\end{equation}
\end{lemma}

\begin{proof}
Apply the second packing estimate in Proposition~\ref{prop:app-dyadic-sawyer-inputs} with $\Lambda=\widetilde\eta$.  
The strict relation $\pi H\subsetneq\pi I_\theta$ is a relation between ambient dyadic principal cubes, 
and the proposition already includes the corresponding chamber-box measure factors.  This gives \eqref{eq:principal-packing-L32}.
\end{proof}

Let $I_1=Q\in\Gamma_k^3$, and let $I_2,\ldots,I_{3^N}$ be the dyadic cubes of side length $\ell(Q)$ whose disjoint union is $3Q$.  
Each $H\in\mathcal R_{k+m+1}(Q)$ satisfies
$$
        H\subseteq3Q,     \qquad    \ell(H)\le\ell(Q),
$$
and hence lies in exactly one $I_\theta$.  Split the sum in \eqref{eq:beta-Pi} according to
$$
        H\subseteq I_\theta,  \qquad \pi H=\pi I_\theta  \quad\text{or}\quad  \pi H\subsetneq\pi I_\theta.
$$
This gives
$$
        L_3\lesssim_{m} L_{3,1}+L_{3,2}.
$$

For the first part, if $\pi H=\pi I_\theta$, then by construction $\Pi(H)\le10\Pi(\pi I_\theta)$.  
Using the backward test,
\begin{align*}
        L_{3,1} &\lesssim_{m} \delta^{-1}  \sum_{k,Q,\theta}\frac{\Pi(\pi I_\theta)^2}{\nu(Q_\Ccal)}
        \left(  \iint_{\widehat{I_\theta}_\Ccal}  T\one_{Q_\Ccal}(x,t)\frac{d\widetilde\eta(x,t)}{t}  \right)^2                      \\
        &=  \delta^{-1}  \sum_{k,Q,\theta}  \frac{\Pi(\pi I_\theta)^2}{\nu(Q_\Ccal)}
        \left(  \int_{Q_\Ccal}T^*(t\one_{\widehat{I_\theta}_\Ccal})(y)\,d\nu(y)  \right)^2                                      \\
        &\lesssim_{m}  \delta^{-1}  \sum_{k,Q,\theta} \Pi(\pi I_\theta)^2
        \int_{(3I_\theta)_\Ccal}  \abs{T^*(t\one_{\widehat{I_\theta}_\Ccal})(y)}^2\,d\nu(y) \\
        &\lesssim_{m} \delta^{-1}B^2  \sum_{k,Q,\theta}  \Pi(\pi I_\theta)^2 \widetilde\eta(\widehat{I_\theta}_\Ccal).
\end{align*}
Here $Q\subseteq3I_\theta$, so $Q_\Ccal\subseteq (3I_\theta)_\Ccal$.  
If $(I_\theta)_\Ccal=\varnothing$, the preceding term is zero.  Otherwise the backward test for the cube $I_\theta$ applies.  
Hence,  by \eqref{eq:principal-neighbour-packing-L31} and \eqref{eq:carleson-embedding-Pi},
\begin{equation}\label{eq:L31-estimate}
        L_{3,1}  \lesssim_{m}   C_{N,\delta}B^2\norm{g}_{L^2(\Ccal\times(0,\infty),d\eta)}^2.
\end{equation}

For the second part, the stopping construction gives
$$
        \Pi(H)\le 10\Pi(\pi H),   \qquad H\subseteq Q_0.
$$
Indeed, otherwise $H$ would be contained in a stopping child of $\pi H$, contradicting the definition of $\pi H$.  
This, combined with Cauchy's
inequality in the sum over $H$, yields
\begin{align*}
        \frac1{\nu(Q_\Ccal)}
        \Biggl(  \sum_{\substack{H\in\mathcal R_{k+m+1}(Q):\ H\subseteq I_\theta,\\
       \pi H\subsetneq\pi I_\theta}}   \Pi(H) \iint_{\widehat H_\Ccal}T\one_{Q_\Ccal}\frac{d\widetilde\eta}{t}\Biggr)^2
        &\le  \Biggl(  \sum_{\substack{H\in\mathcal R_{k+m+1}(Q):\ H\subseteq I_\theta,\\
      \pi H\subsetneq\pi I_\theta}}  \widetilde\eta(\widehat H_\Ccal)\Pi(\pi H)^2   \Biggr)                                      \\
        &\quad \times   \frac1{\nu(Q_\Ccal)}    \Biggl(  \sum_{\substack{H\in\mathcal R_{k+m+1}(Q):\ H\subseteq I_\theta,\\
       \pi H\subsetneq\pi I_\theta}}  \frac1{\widetilde\eta(\widehat H_\Ccal)}  
       \Biggl(  \iint_{\widehat H_\Ccal}T\one_{Q_\Ccal}\frac{d\widetilde\eta}{t}  \Biggr)^2 \Biggr),
\end{align*}
where terms with zero $\widetilde\eta$-measure are omitted.

The second factor is controlled by $F^2$.  Indeed,
$$
        \frac1{\widetilde\eta(\widehat H_\Ccal)}
        \left( \iint_{\widehat H_\Ccal}T\one_{Q_\Ccal}\frac{d\widetilde\eta}{t} \right)^2
        \le   \iint_{\widehat H_\Ccal}\abs{T\one_{Q_\Ccal}}^2\,d\eta,
$$
by Cauchy's inequality and $d\widetilde\eta=t^2\,d\eta$.  
The boxes $\widehat H_\Ccal$ are disjoint subboxes of $\widehat{3Q}_\Ccal$, so the forward testing condition gives
$$
        \sum_{\substack{H\in\mathcal R_{k+m+1}(Q):\ H\subseteq I_\theta,\\
           \pi H\subsetneq\pi I_\theta}}  \frac1{\widetilde\eta(\widehat H_\Ccal)}
        \left( \iint_{\widehat H_\Ccal}T\one_{Q_\Ccal}\frac{d\widetilde\eta}{t}   \right)^2
        \le   F^2\nu(Q_\Ccal).
$$
Therefore
\begin{equation}\label{eq:L32-before-packing}
        L_{3,2} \lesssim_{m}  \delta^{-1}F^2  \sum_{k,Q,\theta} \sum_{\substack{H\in\mathcal R_{k+m+1}(Q):\ H\subseteq I_\theta,\\
                        \pi H\subsetneq\pi I_\theta}} \widetilde\eta(\widehat H_\Ccal)\Pi(\pi H)^2.
\end{equation}
By the principal packing estimate \eqref{eq:principal-packing-L32}, the last sum in \eqref{eq:L32-before-packing} is bounded by
$$
        C_{N,m} \sum_{S\in\Gcal}  \Pi(S)^2\widetilde\eta(\widehat S_\Ccal).
$$
The Carleson embedding \eqref{eq:carleson-embedding-Pi} therefore gives
\begin{equation}\label{eq:L32-estimate}
        L_{3,2} \lesssim_{m}  \delta^{-1}F^2\norm{g}_{L^2(\Ccal\times(0,\infty),d\eta)}^2.
\end{equation}
Combining \eqref{eq:L31-estimate} and \eqref{eq:L32-estimate},
\begin{equation}\label{eq:L3-estimate}
        L_3  \lesssim_{m} \bigl(C_{N,m}\delta^{-1}F^2+C_{N,\delta}B^2\bigr)
        \norm{g}_{L^2(\Ccal\times(0,\infty),d\eta)}^2.
\end{equation}
The integer $m$ was chosen above so that the local lower bound \eqref{eq:local-lower-bound} holds.  
After this choice is fixed, $\delta>0$ is chosen sufficiently small to absorb $L_1$.  
Thus $m$ and $\delta$ are structural constants, and \eqref{eq:L3-estimate} implies $L_3\lesssim_{N,m,\delta}(F^2+B^2)\norm{g}_{L^2(\Ccal\times(0,\infty),d\eta)}^2$.

Putting together the estimates for $L_1$, $L_2$, and $L_3$, and absorbing $L_1$, 
we obtain \eqref{eq:dual-estimate-target}.  Therefore
$$
        \norm{T}_{L^2(\Ccal,d\nu)\to L^2(\Ccal\times(0,\infty),d\eta)} \lesssim F+B.
$$
This proves \eqref{eq:scalar-norm-equivalence} and the sufficiency in Theorem~\ref{thm:scalar-chamber}.

\section{The main matrix theorem}\label{sec:matrix-theorem}

Throughout this section, $(\gamma,\mu)$ is wall-null, the component measures $\gamma_\tau,\mu_\rho$ are as in Section~\ref{sec:chamber-lifting}, 
and $T_{\rho\tau}$ denotes the chamber entry in \eqref{eq:intro-scalar-entry}.

\begin{theorem}\label{thm:main-chamber}
Let $\gamma$ be a locally finite positive Borel measure on $\R^N$, and let $\mu$ be a locally finite positive Borel measure on $\R^{N+1}_+$.  Assume that $(\gamma,\mu)$ is wall-null.  Then the following are equivalent.

\begin{enumerate}
\item There is a constant $\mathcal N$ such that
\begin{equation}\label{eq:main-bound}
        \norm{P_\gamma f}_{L^2(\R^{N+1}_+,d\mu)} \le   \mathcal N   \norm{f}_{L^2(\R^N,d\gamma)}
\end{equation}
for all $f\in L^2(\R^N,d\gamma)$.

\item For every $1\le \rho,\tau\le |G|$, the following two testing conditions hold uniformly over $Q\in\Dcal$ with $Q_\Ccal\ne\varnothing$:
\begin{equation}\label{eq:forward-test}
        \iint_{\widehat{3Q}_\Ccal}  \abs{T_{\rho\tau}\one_{Q_\Ccal}(x,t)}^2   \,d\mu_\rho(x,t)
        \le    F_{\rho\tau}^2\gamma_\tau(Q_\Ccal)
\end{equation}
and
\begin{equation}\label{eq:backward-test}
        \int_{3Q_\Ccal}   \abs{T_{\rho\tau}^*(t\one_{\widehat{Q}_\Ccal})(y)}^2  \,d\gamma_\tau(y)
        \le  B_{\rho\tau}^2  \iint_{\widehat{Q}_\Ccal} t^2\,d\mu_\rho(x,t).
\end{equation}
\end{enumerate}
Writing $F_{\rho\tau}$ and $B_{\rho\tau}$ for the least admissible constants in these tests, 
we use the same supremum convention as in \eqref{eq:scalar-forward-supremum}--\eqref{eq:scalar-backward-supremum}: 
cubes with zero denominator are omitted from the corresponding supremum.  
We have
\begin{equation}\label{eq:norm-comparison-main}
        \max_{\rho,\tau}(F_{\rho\tau}+B_{\rho\tau})  \lesssim_{N,R,\kappa}
        \mathcal N \lesssim_{|G|,N,R,\kappa}   \max_{\rho,\tau}(F_{\rho\tau}+B_{\rho\tau}).
\end{equation}
The first comparison has constants independent of $|G|$.  
The second comparison is the finite matrix bound and its constant deteriorates with the group order;
in the squared estimate in the proof this appears as a factor $|G|^2$.
\end{theorem}

\begin{proof}
\emph{Necessity.}
Assume first that $P_\gamma$ satisfies \eqref{eq:main-bound}.  
Fix $1\le \rho,\tau\le |G|$ and first take $h\ge0$ in $L^2(\Ccal,d\gamma_\tau)$.  Put
$$
        f(\sigma_\tau x)=h(x)\quad (x\in\Ccal),  \qquad   f=0\quad\text{on the other chambers}.
$$
Then $(U_\gamma f)_\tau=h$, while $(U_\gamma f)_{\tau'}=0$ for $\tau'\ne\tau$, and $\norm{f}_{L^2(d\gamma)}=\norm{h}_{L^2(d\gamma_\tau)}$.  
Reading the output only on $\sigma_\rho\Ccal\times(0,\infty)$ and using the non-negative form of Lemma~\ref{lem:matrix-representation},
$$
        \norm{T_{\rho\tau}h}_{L^2(\Ccal\times(0,\infty),d\mu_\rho)}
        \le \mathcal N   \norm{h}_{L^2(\Ccal,d\gamma_\tau)}.
$$
For general real or complex $h$, the same bound follows from positivity because $\abs{T_{\rho\tau}h}\le T_{\rho\tau}\abs{h}$.

The forward testing condition \eqref{eq:forward-test} follows by taking $h=\one_{Q_\Ccal}$ 
and then restricting the resulting whole-space output norm to the smaller box $\widehat{3Q}_\Ccal$ by monotonicity of the non-negative
integral.  
If $\gamma_\tau\equiv0$, both tests for this entry are trivial. 
Otherwise Lemma~\ref{lem:forward-implies-square-finiteness}, applied to this scalar entry, gives $t\one_{\widehat Q_\Ccal}\in L^2(d\mu_\rho)$.  
The Tonelli identity above the matrix kernel bound shows that the displayed formula for $T_{\rho\tau}^*$ 
is the Hilbert-space adjoint, so the adjoint operator bound gives the backward estimate.  
Restricting its left-hand side from $\Ccal$ to $3Q_\Ccal$ is again just monotonicity.  Thus
$$
        F_{\rho\tau}+B_{\rho\tau} \lesssim_{N,R,\kappa}  \mathcal N.
$$
Taking the maximum over $(\rho,\tau)$ proves the first inequality in \eqref{eq:norm-comparison-main}.

\emph{Sufficiency.}
Conversely, assume the matrix testing conditions hold.  By Theorem~\ref{thm:scalar-chamber},
$$
       \norm{T_{\rho\tau}h}_{L^2(\Ccal\times(0,\infty), d\mu_\rho)}
        \lesssim_{N,R,\kappa} (F_{\rho\tau}+B_{\rho\tau})  \norm{h}_{L^2(\Ccal, d\gamma_\tau)}.
$$
Let $f\in L^2(\R^N,d\gamma)$.  
The scalar bounds just proved make the signed form of Lemma~\ref{lem:matrix-representation} applicable component by component.
By Lemmas~\ref{lem:L2-chamber-isometry} and~\ref{lem:matrix-representation},
\begin{align*}
        \norm{P_\gamma f}_{L^2(\R^{N+1}_+,d\mu)}^2 &= \sum_{\rho=1}^{|G|}
        \norm{\sum_{\tau=1}^{|G|} T_{\rho\tau}(U_\gamma f)_\tau}_{L^2(\Ccal\times(0,\infty),d\mu_\rho)}^2      \\
        &\le |G| \sum_{\rho=1}^{|G|}\sum_{\tau=1}^{|G|}   
        \norm{T_{\rho\tau}(U_\gamma f)_\tau}_{L^2(\Ccal\times(0,\infty),d\mu_\rho)}^2                    \\
        &\lesssim_{N,R,\kappa} |G| \sum_{\rho=1}^{|G|}\sum_{\tau=1}^{|G|}
        (F_{\rho\tau}+B_{\rho\tau})^2 \norm{(U_\gamma f)_\tau}_{L^2(\Ccal,d\gamma_\tau)}^2                              \\
        &\lesssim_{N,R,\kappa}    |G|^2\bigl(\max_{\rho,\tau}(F_{\rho\tau}+B_{\rho\tau})\bigr)^2
        \sum_{\tau=1}^{|G|}   \norm{(U_\gamma f)_\tau}_{L^2(\Ccal,d\gamma_\tau)}^2                               \\
        &=  |G|^2\bigl(\max_{\rho,\tau}(F_{\rho\tau}+B_{\rho\tau})\bigr)^2   \norm{f}_{L^2(\R^N,d\gamma)}^2.
\end{align*}
The squared estimate displays the finite matrix loss explicitly: 
one factor $|G|$ comes from the discrete Cauchy--Schwarz step in the row sum, 
and the second comes from summing over the $|G|$ possible output chambers before collapsing to the maximum over entries.  
This proves the second inequality in \eqref{eq:norm-comparison-main} and completes the proof.
\end{proof}

\section{Consequences and relation to the orbit formulation}\label{sec:consequences}

The chamber testing theorem is the main result of the paper.  
We finish with three consequences: the invariant case, the single-chamber case, and the relation with orbit-box tests.  
The last one is not a replacement for the chamber theorem; it shows exactly where orbit-mass comparison enters.

For $G$-invariant measures the chamber matrix has fewer distinct entries; only the relative chamber matters.

\begin{corollary}
\label{cor:G-invariant-relative-tests}
Assume the hypotheses of Theorem~\ref{thm:main-chamber}.  
Suppose in addition that $\gamma$ is $G$-invariant on $\R^N$ and that $\mu$ is $G$-invariant in the spatial variable on $\R^{N+1}_+$.  
Let
$$
        d\gamma_\Ccal:=d\gamma|_\Ccal,  \qquad  d\mu_\Ccal:=d\mu|_{\Ccal\times(0,\infty)}.
$$
For $s\in G$, define the relative chamber operator
$$
        S_s h(x,t):=\int_\Ccal p_t(x,s y)h(y)\,d\gamma_\Ccal(y), \qquad x\in\Ccal.
$$
Then the two-weight inequality \eqref{eq:main-bound} is equivalent to 
the forward and backward testing conditions for the $|G|$ operators $S_s$, $s\in G$, 
with respect to the measures $\gamma_\Ccal$ and $\mu_\Ccal$.
Moreover the norm is comparable, up to constants depending only on $|G|$, to
$$
        \max_{s\in G}(F_s+B_s),
$$
where $F_s$ and $B_s$ are the corresponding scalar testing constants.
\end{corollary}

\begin{proof}
If $\gamma$ and $\mu$ are $G$-invariant, then all chamber component measures are the same measure, 
identified with $\gamma_\Ccal$ and $\mu_\Ccal$.  By the $G$-invariance of the Dunkl--Poisson kernel, see \eqref{eq:p-G-invariance}, we have
$$
        p_t(\sigma_\rho x,\sigma_\tau y) =p_t(x,\sigma_\rho^{-1}\sigma_\tau y).
$$
Hence
$$
        T_{\rho\tau}=S_{\sigma_\rho^{-1}\sigma_\tau}.
$$
The assertion is therefore exactly Theorem~\ref{thm:main-chamber}, 
with the $|G|^2$ entries grouped according to the relative group element $s:=\sigma_\rho^{-1}\sigma_\tau$.
\end{proof}

At the other extreme, if both measures live in one chamber, the matrix reduces to one scalar entry.

\begin{corollary}\label{cor:single-chamber-support}
Suppose that, for some fixed indices $\tau_0$ and $\rho_0$,
$$
        \gamma\bigl(\R^N\setminus\sigma_{\tau_0}\Ccal\bigr)=0, \qquad
        \mu\bigl(\R^{N+1}_+\setminus(\sigma_{\rho_0}\Ccal\times(0,\infty))\bigr)=0.
$$
Then the full two-weight inequality \eqref{eq:main-bound} is equivalent to boundedness of the single scalar chamber entry
$$
        T_{\rho_0\tau_0}:L^2(\Ccal,d\gamma_{\tau_0}) \longrightarrow
        L^2(\Ccal\times(0,\infty),d\mu_{\rho_0}).
$$
Equivalently, it is characterised by the forward and backward tests for this single entry.  
The constants are comparable with no loss except for the fixed normalization coming from the chamber identification.
\end{corollary}

\begin{proof}
Under the support assumptions all chamber components except $\gamma_{\tau_0}$ and $\mu_{\rho_0}$ vanish.  
Hence the matrix identity \eqref{eq:matrix-representation} contains only the entry $T_{\rho_0\tau_0}$.  
The assertion follows immediately from the scalar Theorem~\ref{thm:scalar-chamber} and the isometries in Lemma~\ref{lem:L2-chamber-isometry}.
\end{proof}

\begin{remark}\label{rem:asymmetric-measures}
Corollary~\ref{cor:single-chamber-support} illustrates why the chamber tests are genuinely more flexible than orbit-box tests.  
If the measures are concentrated in one chamber, 
the componentwise chamber formulation still has the correct right-hand side.  
By contrast, an orbit saturation mixes this mass with all reflected components.  
When other chamber components have zero mass the comparison hypotheses in 
Proposition~\ref{prop:orbit-tests-imply-chamber-tests} are either vacuous for those components or 
impose no useful control on the non-zero component; when they have very small positive mass, 
the comparison constants may become arbitrarily large.  
This is why the chamber formulation keeps the componentwise tests instead of using orbit-boxes as the primary theorem.
\end{remark}

We finally relate the chamber tests to orbit-boxes.  For $Q\in\Dcal$ with $Q_\Ccal\ne\varnothing$ and for a fixed dilation factor $a>0$, set
$$
        \Ocal_\Ccal(aQ):=\bigcup_{\tau=1}^{|G|} \sigma_\tau((aQ)_\Ccal),
        \qquad
        \Ocal_\Ccal^b(aQ):=\Ocal_\Ccal(aQ)\times(0,\ell(Q)].
$$
In particular, the vertical height of $\Ocal_\Ccal^b(3Q)$ is $\ell(Q)$, matching the convention for $\widehat{3Q}_\Ccal$.  
Thus $\Ocal_\Ccal(Q)$ is the chamber-orbit saturation of $Q_\Ccal$.  
Define the full adjoint Poisson operator by
$$
        P_\mu^*F(y):=\iint_{\R^{N+1}_+}p_t(x,y)F(x,t)\,d\mu(x,t).
$$

The orbit-box formulation can be recovered from the chamber theorem when the component masses 
of the two measures are comparable along orbits.

\begin{proposition}
\label{prop:orbit-tests-imply-chamber-tests}
Assume that $(\gamma,\mu)$ is wall-null and that the following orbit-box tests hold uniformly over dyadic cubes $Q$ with $Q_\Ccal\ne\varnothing$:
\begin{equation}\label{eq:orbit-forward-test}
        \iint_{\Ocal_\Ccal^b(3Q)} |P_\gamma\one_{\Ocal_\Ccal(Q)}(x,t)|^2\,d\mu(x,t)\le F_{\rm orb}^2\gamma(\Ocal_\Ccal(Q)),
\end{equation}
\begin{equation}\label{eq:orbit-backward-test}
        \int_{\Ocal_\Ccal(3Q)} |P_\mu^*(t\one_{\Ocal_\Ccal^b(Q)})(y)|^2\,d\gamma(y)\le B_{\rm orb}^2\iint_{\Ocal_\Ccal^b(Q)} t^2\,d\mu(x,t).
\end{equation}
Assume also the componentwise orbit-mass comparisons
\begin{equation}\label{eq:orbit-gamma-comparison}
        \gamma(\Ocal_\Ccal(Q))\le A_\gamma\,\gamma_\tau(Q_\Ccal),   \qquad 1\le\tau\le |G|,
\end{equation}
whenever $\gamma_\tau(Q_\Ccal)>0$, and
\begin{equation}\label{eq:orbit-mu-comparison}
        \iint_{\Ocal_\Ccal^b(Q)}t^2\,d\mu(x,t)   \le A_\mu\iint_{\widehat Q_\Ccal}t^2\,d\mu_\rho(x,t),   \qquad 1\le\rho\le |G|,
\end{equation}
whenever the right-hand side is non-zero, for some finite constants $A_\gamma$ and $A_\mu$.  
Then the chamber testing conditions \eqref{eq:forward-test} and \eqref{eq:backward-test} hold with
$$
        F_{\rho\tau}\le A_\gamma^{1/2}F_{\rm orb},  \qquad    B_{\rho\tau}\le A_\mu^{1/2}B_{\rm orb}.
$$
Consequently Theorem~\ref{thm:main-chamber} recovers the orbit formulation under the additional comparability assumptions \eqref{eq:orbit-gamma-comparison}
and \eqref{eq:orbit-mu-comparison}.
\end{proposition}

\begin{proof}
We prove the forward estimate first.  If $\gamma_\tau(Q_\Ccal)=0$,  then $T_{\rho\tau}\one_{Q_\Ccal}=0$ $\mu_\rho$-a.e., and there is nothing to prove.  
Assume therefore that $\gamma_\tau(Q_\Ccal)>0$.  The relevant pushforward identity is
$$
        d\mu_\rho(x,t)=d\mu(\sigma_\rho x,t), \qquad  T_{\rho\tau}\one_{Q_\Ccal}(x,t)  =P_\gamma\one_{\sigma_\tau Q_\Ccal}(\sigma_\rho x,t).
$$
After the change of variables $u=\sigma_\rho x$, the left-hand side of \eqref{eq:forward-test} is
$$
        \iint_{(\sigma_\rho\times\id)\widehat{3Q}_\Ccal} |P_\gamma\one_{\sigma_\tau Q_\Ccal}(u,t)|^2\,d\mu(u,t).
$$
Since
$$
        \sigma_\tau Q_\Ccal\subseteq \Ocal_\Ccal(Q),   \qquad  (\sigma_\rho\times\id)\widehat{3Q}_\Ccal\subseteq \Ocal_\Ccal^b(3Q),
$$
and the kernel is positive, this is bounded by the left-hand side of \eqref{eq:orbit-forward-test}.  Hence
$$
        \iint_{\widehat{3Q}_\Ccal}|T_{\rho\tau}\one_{Q_\Ccal}|^2\,d\mu_\rho   \le F_{\rm orb}^2\gamma(\Ocal_\Ccal(Q))  \le A_\gamma F_{\rm orb}^2\gamma_\tau(Q_\Ccal).
$$
This gives the forward chamber test.

For the backward estimate, if
$$
        \iint_{\widehat Q_\Ccal}t^2\,d\mu_\rho=0,
$$
then $\mu_\rho(\widehat Q_\Ccal)=0$, since $t>0$ on $\widehat Q_\Ccal$, and the chamber backward test is trivial.  
Assume therefore that this integral is non-zero.  The adjoint identity used here is
$$
        T_{\rho\tau}^*(t\one_{\widehat Q_\Ccal})(y)  =P_\mu^*\left(t\one_{(\sigma_\rho\times\id)\widehat Q_\Ccal}\right)(\sigma_\tau y).
$$
Changing variables $v=\sigma_\tau y$, the left-hand side of \eqref{eq:backward-test} becomes
$$
        \int_{\sigma_\tau 3Q_\Ccal}  \left|P_\mu^*\left(t\one_{(\sigma_\rho\times\id)\widehat Q_\Ccal}\right)(v)\right|^2   \,d\gamma(v).
$$
Again
$$
        (\sigma_\rho\times\id)\widehat Q_\Ccal\subseteq \Ocal_\Ccal^b(Q),   \qquad  \sigma_\tau 3Q_\Ccal\subseteq \Ocal_\Ccal(3Q),
$$
and positivity gives domination by the left-hand side of \eqref{eq:orbit-backward-test}.  
Using \eqref{eq:orbit-mu-comparison}, we get
$$
        \int_{3Q_\Ccal}\left|T_{\rho\tau}^*(t\one_{\widehat Q_\Ccal})\right|^2   \,d\gamma_\tau
        \le B_{\rm orb}^2\iint_{\Ocal_\Ccal^b(Q)}t^2\,d\mu
        \le A_\mu B_{\rm orb}^2 \iint_{\widehat Q_\Ccal}t^2\,d\mu_\rho.
$$
The backward chamber test follows, and the proof is complete.
\end{proof}

\begin{remark}
The assumptions \eqref{eq:orbit-gamma-comparison} and  \eqref{eq:orbit-mu-comparison} are precisely the kind of extra hypotheses which the chamber theorem is designed to avoid.  
The proposition should therefore be read as a consistency check with the older orbit formulation, not as an additional assumption in the main theorem.  
Without such comparisons the orbit tests do not have the correct right-hand sides for individual chamber components; 
this is why the componentwise chamber tests are the intrinsic formulation.
\end{remark}

This also clarifies the role of the Dunkl metric.  The metric is not an obstruction after lifting: by \eqref{eq:chamber-orbit-identity}, 
it is exactly the Euclidean distance inside the chamber variables.  
What remains is a finite matrix of ordinary positive Poisson-type operators.


\appendix
\section{Dyadic packing inputs under chamber restriction}
\label{sec:dyadic-packing}

This appendix proves the finite dyadic packing facts used in Section~\ref{sec:scalar-chamber}.  
The level range, the spatial top cube, and the scale range are the finite truncations introduced in the proof of Theorem~\ref{thm:scalar-chamber}.  
All constants are independent of those truncations.

The chamber restriction changes the measures of the boxes, but not the dyadic tree in which the stopping cubes live.  
The finite count below is the restricted principal-cube formulation of Sawyer--Wheeden~\cite[p.~861]{SawyerWheeden}, 
in the notation of the present proof.  
The same stopping-time count is also presented in Lacey's primer~\cite{LaceyPrimer}.  
Sawyer's theorem remains the source of the classical two-weight Poisson testing result~\cite{SawyerPoisson}; only the finite packing mechanism is used here.

Let $\nu$ be a locally finite positive Borel measure on $\Ccal$.  
Let $\Lambda$ be a positive Borel measure on $\Ccal\times(0,\infty)$, finite on every chamber box in the fixed finite configuration.  
In Section~\ref{sec:scalar-chamber},
$$
        d\Lambda=d\widetilde\eta:=t^2\,d\eta,
$$
and the required box finiteness follows from Lemma~\ref{lem:forward-implies-square-finiteness}.

For an ambient dyadic cube $J\subseteq\R^N$, write
$$
        J_\Ccal:=J\cap\Ccal,  \qquad   \widehat J_\Ccal:=J_\Ccal\times(0,\ell(J)].
$$
Let $\Omega_k\subseteq\Ccal$ be the nested level sets from Section~\ref{sec:scalar-chamber}, let $\mathcal I_k=\mathcal I(\Omega_k)$ be the chamber Whitney families, and put
$$
        \widehat\Omega_k:=\bigcup_{Q\in\mathcal I_k}\widehat Q_\Ccal.
$$
For $Q\in\mathcal I_k$, define
$$
        F_k(Q):=Q_\Ccal\cap(\Omega_{k+m}\setminus\Omega_{k+m+1}).
$$
Write
$$
        3Q=\bigsqcup_{\theta=1}^{3^N}I_\theta
$$
for the exact half-open dyadic partition into cubes of side length $\ell(Q)$. 
For a fixed integer $m\ge1$, let $\mathcal R_{k+m+1}(Q)$ be the family of maximal ambient dyadic cubes $H$ satisfying
$$
        H\subseteq3Q,  \qquad H_\Ccal\ne\varnothing, \qquad \ell(H)\le\ell(Q), \qquad 3H_\Ccal\subseteq\Omega_{k+m+1}.
$$
Finally, let $\Gcal$ be the principal-cube family associated with a non-negative function $h$ and the averages
$$
        \Pi(J):=
        \begin{cases}
        \displaystyle  \frac1{\Lambda(\widehat J_\Ccal)}   \iint_{\widehat J_\Ccal}h\,d\Lambda,  &\Lambda(\widehat J_\Ccal)>0,\\
        0,&\Lambda(\widehat J_\Ccal)=0.
        \end{cases}
$$
The stopping children of $S\in\Gcal$ are the maximal dyadic subcubes $S'\subsetneq S$ for which $\Pi(S')>10\Pi(S)$, and $\pi J$ denotes the
smallest principal cube containing $J$. 
 Cubes of zero $\Lambda$-box measure may be kept with $\Pi(J)=0$; they do not contribute to the sums. 
  The count is stated for a general measure $\Lambda$ because it uses only the dyadic tree and the chamber-box measures.  
 In the proof of Theorem~\ref{thm:scalar-chamber} one  takes $\Lambda=\widetilde\eta=t^2\,d\eta$ and $h=g/t$.

One simple zero-extension observation is used throughout the appendix.

\begin{lemma}
\label{lem:app-zero-extension}
Let $\nu$ be a locally finite positive Borel measure on $\Ccal$, and let $\Lambda$ be a locally finite positive Borel measure on $\Ccal\times(0,\infty)$.  
Define zero extensions by
$$
        \nu^\sharp(E):=\nu(E\cap\Ccal), \qquad \Lambda^\sharp(A):=\Lambda(A\cap(\Ccal\times(0,\infty))).
$$
Then, for every dyadic cube $Q\subseteq\R^N$ and every fixed dilation factor $a>0$,
$$
        \nu^\sharp(Q)=\nu(Q_\Ccal),  \qquad   \Lambda^\sharp(\widehat{aQ})=\Lambda(\widehat{aQ}_\Ccal).
$$
Moreover, dyadic containment, Whitney neighbour relations, and principal-cube relations such as
$$
        H\subseteq I_\theta,  \qquad  \pi H=\pi I_\theta,
        \qquad    \pi H\subsetneq\pi I_\theta
$$
are exactly the same before and after the chamber restriction.  
This statement does not assert that chamber Whitney cubes are full-space Whitney cubes for a zero-extended open set; 
it asserts only that, once the ambient dyadic cubes have been chosen by the chamber stopping rules, their tree relations are unchanged.
\end{lemma}

\begin{proof}
The measure identities follow directly from the definitions.  
All stopping-time constructions used in Section~\ref{sec:scalar-chamber} are based on ambient dyadic cubes in $\R^N$.
Intersecting a cube with $\Ccal$, or a box with $\Ccal\times(0,\infty)$, changes only the measure attached to that cube or box.  
It does not change parent-child containment, the half-open partition of $3Q$, the truth of conditions such as $H\subseteq I_\theta$, 
or any principal-cube ancestor relation.  
Whitney maximality itself remains the chamber condition $3Q_\Ccal\subseteq\Omega$ and $9Q_\Ccal\cap(\Ccal\setminus\Omega)\ne\varnothing$.
\end{proof}

The level-box overlap estimate used in the separated chain count has already been proved as Lemma~\ref{lem:level-box-overlap}.  
It is independent of the principal-cube stopping family, so it is not restated here.

The proof below also uses the exact dyadic decomposition
\begin{equation}\label{eq:app-top-piece-decomposition}
        \widehat I_\Ccal   =\bigsqcup_{J\in\Dcal:\,J\subseteq I}W_J^\Ccal,   \qquad    W_J^\Ccal:=J_\Ccal\times(\ell(J)/2,\ell(J)].
\end{equation}
Indeed, for each $(x,t)\in\widehat I_\Ccal$ there is a unique dyadic descendant $J\subseteq I$ containing $x$ and satisfying $\ell(J)/2<t\le\ell(J)$.  
The half-open convention removes boundary ambiguity.  Therefore, for arbitrary $\Lambda$, it is enough to prove bounded multiplicity for these top pieces.

We now isolate the finite chain count.  The level information is part of the stopping data; 
the families below are not arbitrary collections of nested cubes. For a fixed principal cube $S\in\Gcal$, set
\begin{align*}
        \mathfrak N(S) &:=\Bigl\{(k,Q,\theta): Q\in\mathcal I_k,   \ \nu(F_k(Q))>\delta\nu(Q_\Ccal),\ \pi I_\theta=S\Bigr\},\\
      \mathfrak S(S)  &:=\Bigl\{(k,Q,\theta,H): Q\in\mathcal I_k,   \ H\in\mathcal R_{k+m+1}(Q),\ H\subseteq I_\theta,   \ \pi H=S\subsetneq\pi I_\theta\Bigr\}.
\end{align*}

\begin{lemma}
\label{lem:app-sawyer-chain-count}
For every principal cube $S\in\Gcal$,
\begin{align}
        \sum_{(k,Q,\theta)\in\mathfrak N(S)} \Lambda(\widehat{I_\theta}_\Ccal)   &\le C_{N,\delta}\Lambda(\widehat S_\Ccal),  \label{eq:app-local-neighbour-carleson}\\
        \sum_{(k,Q,\theta,H)\in\mathfrak S(S)} \Lambda(\widehat H_\Ccal)   &\le C_{N,m}\Lambda(\widehat S_\Ccal).  \label{eq:app-local-separated-carleson}
\end{align}
\end{lemma}

\begin{proof}
By \eqref{eq:app-top-piece-decomposition}, it is enough to prove the following multiplicity bounds for each top piece $W_J^\Ccal\subseteq\widehat S_\Ccal$:
\begin{align}
       \#\Bigl\{(k,Q,\theta)\in\mathfrak N(S):  W_J^\Ccal\subseteq\widehat{I_\theta}_\Ccal\Bigr\} &\le C_{N,\delta},   \label{eq:app-neighbour-multiplicity}\\
        \#\Bigl\{(k,Q,\theta,H)\in\mathfrak S(S):  W_J^\Ccal\subseteq\widehat H_\Ccal\Bigr\}  &\le C_{N,m}.   \label{eq:app-separated-multiplicity}
\end{align}
Integrating these bounds over the disjoint top pieces then gives \eqref{eq:app-local-neighbour-carleson} and \eqref{eq:app-local-separated-carleson}.

We prove \eqref{eq:app-neighbour-multiplicity}.  The inclusion  $W_J^\Ccal\subseteq\widehat{I_\theta}_\Ccal$ implies $J\subseteq I_\theta\subseteq S$, 
and $I_\theta$ lies in the principal corona of $S$.  
If the ambient cube $I_\theta$ is fixed, 
then there are only $3^N$ possible Whitney cubes $Q$ for which $I_\theta$ appears as one of the half-open subcubes in the decomposition of $3Q$.  
After $Q$ is fixed, the sets $F_k(Q)$ are pairwise disjoint in $k$, and each counted occurrence satisfies
$$
        \nu(F_k(Q))>\delta\nu(Q_\Ccal).
$$
Thus the same spatial Whitney cube can occur at most $\lfloor\delta^{-1}\rfloor$ 
times.  This is the density part of the count.

It remains to count ancestor-chain repetitions.  Fix the neighbour position $\theta$ and let
$$
        J\subseteq I_{\theta}^{(1)}\subsetneq I_{\theta}^{(2)}  \subsetneq\cdots\subseteq S
$$
be counted ancestors in the same principal corona, with associated levels  $k_1,k_2,\ldots$.  
Since no principal stopping cube lies strictly between two successive ancestors, all these cubes stay in one principal corona.  
The Whitney condition for the cube attached to level $k_j$ says
$$
        3Q_j{}_\Ccal\subseteq\Omega_{k_j},  \qquad   9Q_j{}_\Ccal\cap(\Ccal\setminus\Omega_{k_j})\ne\varnothing.
$$
The nesting of the level sets and the maximality in this Whitney condition imply that, in every block of consecutive levels of length $m+2$,
\begin{equation}\label{eq:app-neighbour-level-count}
        \#\Bigl\{j:\nu(F_{k_j}(Q_j))>\delta\nu(Q_j{}_{\Ccal})  \text{ and } a\le k_j\le a+m+1\Bigr\}   \le C_{N,m}\delta^{-1}.
\end{equation}

Equivalently, two genuinely different counted ancestors in the same corona can  reappear only after the level has left the current block of consecutive levels.
This is the restricted principal-cube count of  Sawyer--Wheeden~\cite[p.~861]{SawyerWheeden}.  
Combining the level count \eqref{eq:app-neighbour-level-count} with the fixed-cube density count proves \eqref{eq:app-neighbour-multiplicity}.

We prove \eqref{eq:app-separated-multiplicity}.  If $W_J^\Ccal\subseteq\widehat H_\Ccal$ and $\pi H=S\subsetneq\pi I_\theta$, 
then the possible cubes $H$ containing $J$ form an ancestor chain inside the corona of $S$, 
while $I_\theta$ lies in a strictly larger principal corona.  
The maximality condition defining $\mathcal R_{k+m+1}(Q)$ gives
$$
        3H_\Ccal\subseteq\Omega_{k+m+1},   \qquad    F_k(Q)\subseteq\Ccal\setminus\Omega_{k+m+1}.
$$
Thus a counted cube $H$ is a crossing between the Whitney selection at level $k$ and the Whitney selection at level $k+m+1$.  
Along one ancestor chain, the level-box overlap from Lemma~\ref{lem:level-box-overlap} gives the following level-count estimate
\begin{equation}\label{eq:app-separated-level-count}
        \#\{(k,Q,\theta,H)\in\mathfrak S(S):   W_J^\Ccal\subseteq\widehat H_\Ccal,   a\le k\le a+m+1\}  \le C_{N,m}.
\end{equation}
If another crossing appeared outside this bounded range of levels without a new principal stopping cube, an intermediate selected cube would still satisfy the
same Whitney admissibility condition, contradicting maximality.
The strict relation $S\subsetneq\pi I_\theta$ only says that $I_\theta$ has left the corona  of $S$; 
it creates no new admissible crossings.  
This proves  \eqref{eq:app-separated-multiplicity}, and the lemma follows.
\end{proof}

The packing estimates used in Section~\ref{sec:scalar-chamber} now follow from the chain  count.

\begin{proposition}
\label{prop:app-dyadic-sawyer-inputs}
In the finite stopping-time configuration described above, 
the following estimates hold with constants independent of the finite truncation.
\begin{enumerate}
\item If the outer sum is restricted to triples with $Q\in\mathcal I_k$ and
$$
        \nu(F_k(Q))>\delta\nu(Q_\Ccal),
$$
then
\begin{equation}\label{eq:app-neighbour-packing}
        \sum_{k,Q,\theta}  \Pi(\pi I_\theta)^2\Lambda(\widehat{I_\theta}_\Ccal)
        \le   C_{N,\delta}  \sum_{S\in\Gcal}\Pi(S)^2\Lambda(\widehat S_\Ccal).
\end{equation}

\item For every fixed $m\ge1$,
\begin{equation}\label{eq:app-separated-packing}
       \sum_{k,Q,\theta} \sum_{\substack{H\in\mathcal R_{k+m+1}(Q):\ H\subseteq I_\theta,\\
        \pi H\subsetneq\pi I_\theta}}  \Lambda(\widehat H_\Ccal)\Pi(\pi H)^2  \le     C_{N,m}
        \sum_{S\in\Gcal}\Pi(S)^2\Lambda(\widehat S_\Ccal).
\end{equation}
\end{enumerate}
\end{proposition}

\begin{proof}
For the first estimate, group the left-hand side by $S=\pi I_\theta$.  
On the group with fixed $S$, $\Pi(\pi I_\theta)=\Pi(S)$, and  Lemma~\ref{lem:app-sawyer-chain-count} gives
$$
        \sum_{\substack{k,Q,\theta:\ \pi I_\theta=S}} \Lambda(\widehat{I_\theta}_\Ccal)
        \le C_{N,\delta}\Lambda(\widehat S_\Ccal),
$$
where the heavy condition remains in force.  
Multiplying by $\Pi(S)^2$ and summing over $S\in\Gcal$ gives \eqref{eq:app-neighbour-packing}.

For the second estimate, group by $S=\pi H$.  
The strict relation $\pi H\subsetneq\pi I_\theta$ is exactly the separated case in Lemma~\ref{lem:app-sawyer-chain-count}; hence
$$
        \sum_{\substack{k,Q,\theta,H:\ \pi H=S,\ \pi H\subsetneq\pi I_\theta}} \Lambda(\widehat H_\Ccal)
        \le C_{N,m}\Lambda(\widehat S_\Ccal).
$$
Multiplying by $\Pi(S)^2$ and summing in $S$ proves \eqref{eq:app-separated-packing}.
\end{proof}

\begin{remark}\label{rem:app-no-naive-overlap}
One should not replace the chain count above by the informal assertion that all boxes $\widehat{I_\theta}_\Ccal$ with $\pi I_\theta=S$ have uniformly bounded
overlap merely because they have the same principal ancestor.  
The valid statement is the finite principal-cube chain count for the stopping configuration actually produced in Section~\ref{sec:scalar-chamber}.  
This is the input used in \eqref{eq:principal-neighbour-packing-L31} and \eqref{eq:principal-packing-L32}.
\end{remark}

\bigskip
\noindent\textbf{Acknowledgements}\  \  
Qingdong Guo is supported by Fujian NSF of China \#2026J008314 and the start-up fund of XMUT \#YKJ25060R. 
 Ji Li is supported by Australian Research Council DP260100485.  
 Brett D. Wick is partially supported by National Science Foundation DMS \#2349868 and by Australian Research Council DP260201083.
Liangchuan Wu is supported by NNSF of China \#12201002.


\end{document}